\documentclass[a4paper,12pt,reqno]{amsart}

\usepackage{amsmath}
\usepackage[utf8]{inputenc}
\usepackage{amsfonts,amssymb,amstext,mathtools,MnSymbol,xcolor}
\usepackage{times,url}
\usepackage{bm}
\usepackage{natbib}

\RequirePackage[colorlinks,citecolor=blue,urlcolor=blue]{hyperref}
\usepackage[plain,noend]{algorithm2e}
\usepackage{geometry}
\newcommand{\ges}{\geqslant}
\newcommand{\les}{\leqslant}
\newcommand{\p}{\mathrm{pr}}
\newcommand{\e}{{E}}
\newcommand{\R}{\mathbb{R}}
\newcommand{\Sp}{\mathbb{S}_+}
\newcommand{\diag}{\mathrm{diag}\,}
\newcommand{\ind}[1]{\mbox{\rm\large 1}_{\{#1\}}}
\newcommand{\norm}[1]{\|#1\|_2}

\def\v{{\varepsilon}}

\theoremstyle{plain}
\newtheorem{theorem}{Theorem}
\newtheorem{proposition}[theorem]{Proposition}
\newtheorem{corollary}[theorem]{Corollary}
\newtheorem{lemma}[theorem]{Lemma}
\newtheorem{remark}[theorem]{Remark}
\newtheorem{algo}[theorem]{Algorithm}

\theoremstyle{definition}
\newtheorem{assumption}{Assumption}

\numberwithin{theorem}{section}

\title[Spherical clustering in detection of groups of concomitant extremes]{Spherical clustering\\ in detection of groups of concomitant extremes}
\author{V.~Fomichov}
\author{J.~Ivanovs}

\address{Aarhus University, Department of Mathematics, Denmark}

\begin{document}

\begin{abstract}
There is growing empirical evidence that spherical $k$-means clustering performs well at identifying groups of concomitant extremes in high dimensions, thereby leading to sparse models.
We provide one of the first theoretical results supporting this approach, but also demonstrate some pitfalls.
Furthermore, we show that an alternative cost function may be more appropriate for identifying concomitant extremes, and it results in a novel spherical $k$-principal-components clustering algorithm.
Our main result establishes a broadly satisfied sufficient condition guaranteeing the success of this method, albeit in a rather basic setting. Finally, we illustrate in simulations that $k$-principal-components outperforms $k$-means
in the difficult case of weak asymptotic dependence within the groups.
%In other cases the two methods often produce almost identical results, which is confirmed by an application to river discharges at 68 locations.
\end{abstract}

\maketitle

\section{Introduction}
Statistical analysis of high-dimensional data typically relies on various sparsity assumptions, unless structural domain knowledge is available. In the study of extremes, sparsity is especially important since the number of extreme observations is small by definition. The recent survey of~\cite{eng2020_review} reviews different notions of sparsity and the available tools in the context of extreme value statistics. Fundamental here is the notion of concomitant extremes~\citep{cha2015,goi2016}, where the focus is on the groups of variables which can be large simultaneously while others are small. Importantly, extremal dependence can be modelled separately in these groups and then combined into a mixture model, thereby leading to dimension reduction.

Mathematically speaking, our interest is in the identification of the lower dimensional faces of $\R^d_+$ charged with mass by the so-called exponent measure. Alternatively, we may look at the faces of the positive simplex charged with mass by the angular (spectral) probability measure. This task, however, is highly non-trivial since only approximate angles coming from a pre-limit distribution, which normally also puts mass on the interior of the simplex, can be obtained in practice.

Various approaches to identification of concomitant extremes and respective maximal sets have been explored in the literature, see~\cite{goi2016,goi2017,sim2018,chi2019,Meyer_Wintenberger_2021,jalalzai2020informative}. One of the basic ideas is to use a certain neighbourhood of a given face and to assign the respective mass to this face. The simplest method~\citep{goi2016} amounts to thresholding the entries of the rescaled extreme observations. Such thresholding, however, often produces a large number of faces with few corresponding observations, see the discussion in~\cite{goi2017} and a numerical illustration in Figure~\ref{fig:elbow} below. Certain grouping of faces is then necessary, see~\cite{chi2019}.

A different approach was proposed by~\cite{cha2015}, and it amounts (ignoring a preliminary dimension reduction step) to spherical $k$-means clustering~\citep{dhi2001} of the approximate angles. Consistency of spherical clustering for a general cost function has been recently established by~\cite{Janssen_Wan}, who also advocated interpreting centroids as `prototypes of extremal dependence'. Each centroid, or cluster, can then be attributed to a certain face.
One way of doing this is to use thresholding to define respective faces, which has been employed in the review of~\cite{eng2020_review} to obtain interpretable results for a real world 68-dimensional dataset. Importantly, this method can be seen as a \emph{reverse threshold-and-group approach}. Finally, our problem is very different from those addressed by subspace clustering techniques for high-dimensional data~\citep[Ch.~15]{clustering_book}, which are popular in computer science.

Clustering in detection of concomitant extremes and prototypes of extremal dependence may seem intuitive to some extent, and there is growing empirical evidence in support of this method. The drawback is that there is no theoretical justification, apart from the consistency result by~\cite{Janssen_Wan}. Building upon this result, we provide some basic theory showing that the spherical clustering procedure must indeed work in certain cases. However, we also identify some pitfalls and show that an alternative cost function may be more appropriate for identification of concomitant extremes. It results in a novel \emph{spherical $k$-principal-components ($k$-pc) clustering} algorithm (a similar approach in a different context was considered in the empirical study of~\cite{Hill_et_al}; see Remark~\ref{rem:Hill} below). The name derives from the fact that the Perron--Frobenius eigenvectors of certain non-negative definite matrices with non-negative entries are used instead of the means in the updates of the algorithm. Interestingly, these matrices are the analogues of the cross-moment matrices used by~\cite{coo2019} and \cite{dre2019} in their principal component analysis of extremal dependence, whereas the bivariate case has been considered by~\cite{lar2012} in a different context. %In addition, it is worth mentioning that both spherical methods, $k$-means and $k$-pc, are connected with the $L^2$-distance, and so the curse of dimensionality does not seem to be a great issue in this setting~\citep{steinbach2004challenges,Hinneburg_et_al_2000}.

Our main results, Theorem~\ref{thm:means} and Theorem~\ref{thm:eigs}, provide sufficient conditions guaranteeing the success of clustering in detection of concomitant extremes, albeit in a rather basic setting, when the corresponding groups of coordinates are disjoint. Importantly, $k$-means fails in some fundamental cases, whereas $k$-pc is much more robust. Our counterexamples provide further intuition about the applicability of both methods. We conclude with a simulation study which confirms our findings and show that the $k$-means algorithm, unlike $k$-pc, has serious problems in the setting where many pairs in the groups of concomitant extremes exhibit weak asymptotic dependence. Arguably, this is one of the most important settings from the applications point of view~\citep{davison2013geostatistics}. In many other cases the results are almost identical, including the 68-dimensional dataset of river discharges and, to a lesser extent, the dietary intake data from~\cite{Janssen_Wan}, which are analysed in Subsection~\ref{subsection_river_discharges} and the Supplementary Material, respectively.

\section{Preliminaries on multivariate extremes}
\label{sec:prelim}
\subsection{The angular distribution}
Here we recall some elementary theory of multivariate extremes needed in this paper, and refer the reader to the monograph~\cite{res2008} and the surveys of~\cite{dav2015,eng2020_review} for further reading.
It is a common practice in extremes to first standardize the marginals and then study extremal dependence.
Thus we assume that a random vector $Y\in\R^d,d\ges 2$, of interest has unit Fr\'echet marginal distributions: $\p(Y_i\les y)=\exp(-1/y)$ for $y>0$, $i=1,\ldots,d$. The fundamental regularity assumption on $Y$, called multivariate regular variation, can be stated as the weak convergence of the normalized $Y$ conditional on its norm being large:
\begin{equation}
\label{eq:MRV}
\left.\frac{Y}{\norm{Y}}\right|\norm{Y}>t\quad\Rightarrow\quad X,\qquad t\to\infty,
\end{equation}
where the distribution of $X$ is called angular or spectral. This distribution is our main object of interest as it characterizes extremal dependence. It must be noted that any norm can be used here, but we choose the Euclidean norm $\norm{\cdot}$ so that $X$ lies in $\Sp^{d-1}=\{x\in\R^d_+:\norm{x}=1\}$, which is convenient when considering angular dissimilarities in Section~\ref{sec:clustering}. Due to marginal standardization, the mean of $X$ must point to the centre of $\Sp^{d-1}$:
\begin{equation}
\label{eq:means}
\e(X_1)=\cdots=\e(X_d)=\mu>0.
\end{equation}
This property has a certain balancing effect which will come in use later.

The random vector $Y$ satisfying~\eqref{eq:MRV} and having unit Fr\'echet marginals is in the max-domain of attraction of a max-stable distribution uniquely specified by the angle~$X$. It is said that $Y$ admits asymptotic independence (complete dependence) if the limiting max-stable vector has mutually independent (completely dependent) components. The distribution of the corresponding angle $X$ is then as follows:
\begin{itemize}
\item[(i)] \emph{Asymptotic independence}: the distribution of $X$ has mass $1/d$ at every standard basis vector.
\item[(ii)] \emph{Asymptotic complete dependence}: deterministic $X=(1,\ldots,1)/\sqrt{d}$.
\end{itemize}

A common measure of pairwise asymptotic dependence is the tail dependence coefficient:
\[
\chi_{ij}=\lim_{t\to\infty} \p(Y_j>t \mid Y_i>t)=\frac{1}{\mu}\e(X_i\wedge X_j)\in [0,1],
\]
where 0 corresponds to the asymptotic independence in (i) and 1 to the complete dependence in (ii). These are the boundary cases also in other senses as demonstrated by the following result.

\begin{lemma}
\label{lem:mu}
The constant $\mu$ in~\eqref{eq:means} satisfies $1/d\les\mu\les 1/\sqrt{d}$, where the lower bound is achieved iff case (i) holds, and the upper bound is achieved iff case (ii) holds.
\end{lemma}

\begin{proof}
Note that $\sum_i X_i\ges\norm{X}=1$ and take expectation to get the lower bound. The upper bound follows from $1=\sum_i \e(X_i^2)\ges d\mu^2$. Equality in the latter readily implies that all $X_i$ are constant, whereas equality in the former implies that $X_iX_j=0$ a.s.\ for all $i\neq j$.
\end{proof}

In applications it is common to standardize the observations using the empirical marginal distribution functions and then to choose threshold $t$ large, but such that sufficiently many approximate angles are obtained for the statistical analysis of extremal dependence.
This is exactly the procedure used by~\cite{Janssen_Wan}, whose consistency results we will employ in the following. The only difference is that we use unit Fr\'echet marginals instead of standard Pareto, but these are tail equivalent and no changes in the theory arise.

\subsection{Concomitant extremes}
In high-dimensional settings it is of crucial importance to identify the non-empty sets of indices $I\subset\{1,\ldots,d\}$ such that $\p(X_i>0\; \forall i\in I, X_j=0\; \forall j\notin I)>0$, and the respective probabilities. We will focus on the maximal sets -- those not included in other such sets.
Let us define the corresponding faces of $\R^d_+$:
\begin{equation}
\label{eq:faces}
F_I=\{x\in\R^d_+:x_j=0\;\forall j\notin I\}.
\end{equation}
According to~\eqref{eq:means} the probabilities $\p(X_i>0)$ are positive  for every $i$, and so every index must be contained in at least one maximal set. The problem at hand is sparse if the cardinality of such sets $I$ is much smaller than $d$ and their number is manageable.

The case (i) of asymptotic independence corresponds to $d$ faces of dimension~1.
A more interesting situation occurs when the indices can be partitioned so that asymptotic independence is present between the (disjoint) groups, but not necessarily within these groups.
Then any prototype of extremal dependence must belong to some group and, ideally, each group should have a prototype, see Section~\ref{sec:concom} for further discussion.
Thus we arrive to a \emph{basic test scenario}, whereas applicability of the clustering method is much more general:

\begin{assumption}
\label{assumption}
There exists $2\les k\les d$ and a partition $(I_1,\ldots,I_k)$ of the index set $\{1,\ldots,d\}$ such that the union of the respective faces $F_{I_1},\ldots,F_{I_k}$ contains the support of the angular measure:
$\p(X\in F_{I_1}\cup\cdots\cup F_{I_k})=1$. Without loss of generality we assume that the indices in $I_i$ are smaller than those in $I_j$ for all $1\les i<j\les k$.
\end{assumption}

It may help to also assume that the above partition is unique, but this is not strictly required.
The corresponding faces $F_{I_1},\ldots,F_{I_k}$ are mutually orthogonal: $u^\top v=0$ for all $u\in F_{I_i}$, $v\in F_{I_j}$ and $i\neq j$, and the only common element is the origin.
Even in this scenario identification of the groups of concomitant extremes $I_1,\ldots,I_k$ using approximate angles may not be straightforward. For a low-dimensional case of $d=3$ and the partition $I_1=\{1,2\},I_2=\{3\}$, Figure~\ref{fig:ex_clust} depicts a sample from the exact angular law and two samples from its approximations as could arise in practice for different threshold levels~$t$.
%We regard this somewhat simplistic scenario of mutually perpendicular faces as a starting point in the analysis of various approaches to identification of concomitant extremes.
In these examples both clustering methods produce similar centroids (crosses for $k$-means and lozenges for $k$-pc) which are close to the respective faces, and thus the latter can be easily identified.

\begin{figure}[!ht]
\centering
%\begin{tabular}{ccc}
\includegraphics[width=0.25\textwidth]{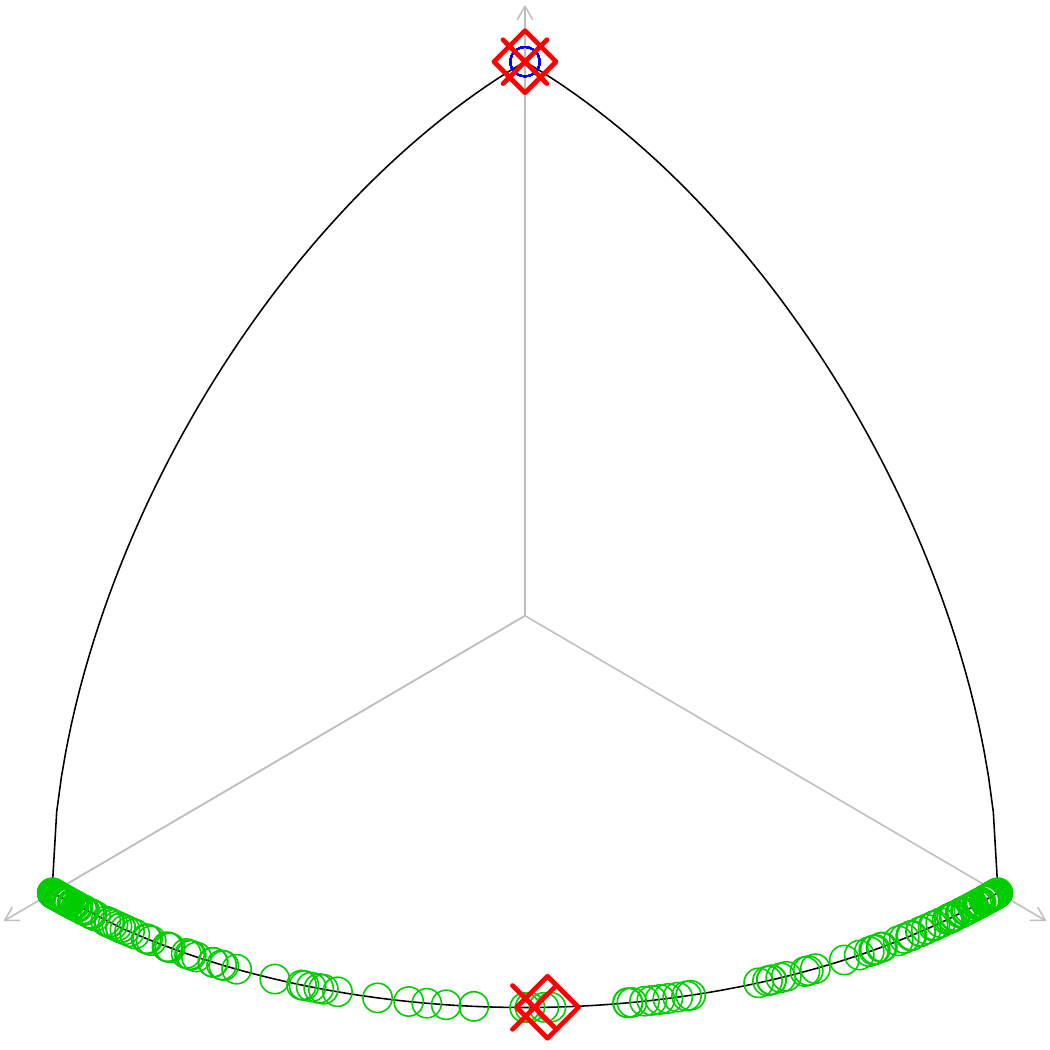}\qquad
\includegraphics[width=0.25\textwidth]{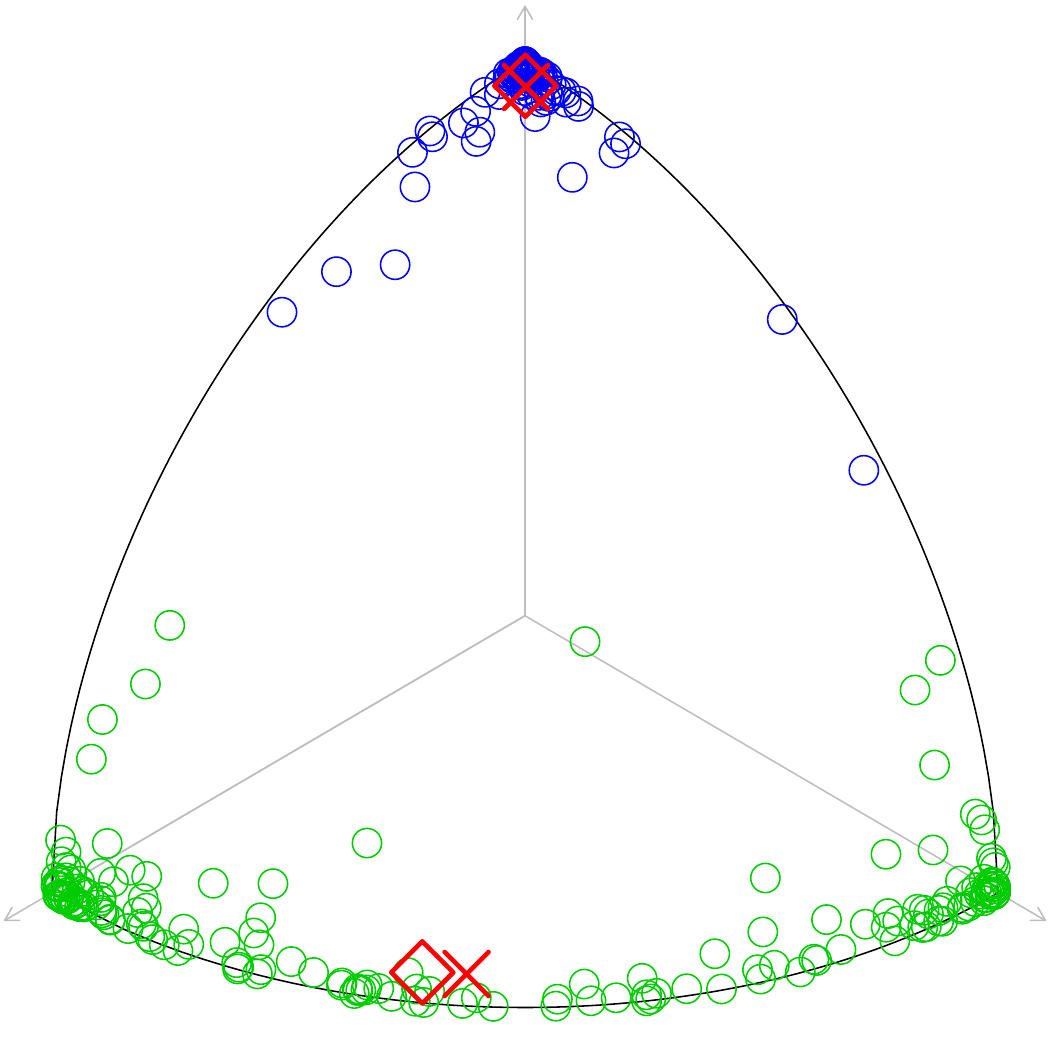}\qquad  %clip, trim=100 100 50 75,
\includegraphics[width=0.25\textwidth]{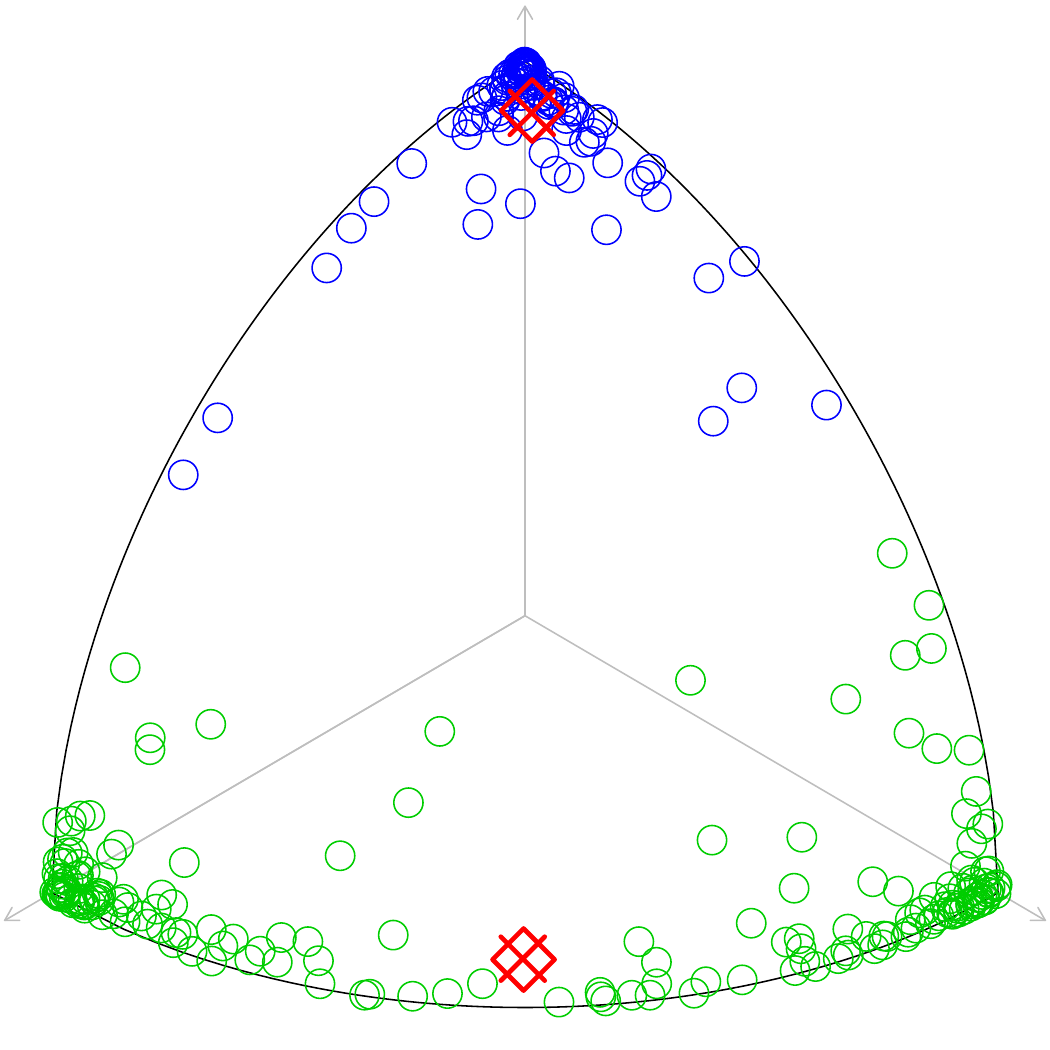}
%(a)&(b)&(c)
%\end{tabular}
\caption{A sample of 300 points in $\Sp^2$ from the exact angular distribution supported by $F_{\{1,2\}}\cup F_{\{3\}}$ and from its approximations. The colours correspond to $2$-means~/~$2$-pc clustering with red crosses~/~lozenges being the centroids.}
\label{fig:ex_clust}
\end{figure}

Let us define the corresponding submodels $X_I$ for $I=I_1,\ldots,I_k$ which have the law of the restriction of $X$ to the indices in $I$ conditional on $\{X\in F_I\}$, and let $p_I=\p(X\in F_I)>0$ be the respective probabilities.
Under Assumption~\ref{assumption} we thus obtain a complete description of the law of $X$ as a mixture model.
Furthermore, $X_I\in\Sp^{|I|-1}$ satisfies~\eqref{eq:means} with $\mu_I=\mu/p_I$ and so it is a legitimate angular model in dimension~$|I|$.
Finally, we define the cross-moment matrices
\begin{equation}
\label{eq:Sigma}
\Sigma_I=\e(X_IX_I^\top),\qquad\Sigma=\e(XX^\top)=\diag(p_{I_1}\Sigma_{I_1},\ldots, p_{I_k}\Sigma_{I_k}),
\end{equation}
where the latter is a block diagonal matrix with $p_I\Sigma_I$ on the diagonal. Note that every $\Sigma_I$ is a non-negative definite matrix with trace $1$, since $\norm{X_I}=1$.
%The matrix $\Sigma$ has been used in extremes before, see Remark~\ref{rem:sigma} and Supplementary Material for exprthe case of H\"usler-Reiss distribution.
An expression of $\Sigma_I$ for the H\"usler-Reiss distribution is given in Supplementary Material, whereas Remark~\ref{rem:sigma} provides connections to the literature on PCA for extremes.
%In  we provide a simple expression of  $\Sigma$ in terms of the expectations of the transformed $(d-1)$-dimensional normal vectors derived from the matrix $\Gamma$ in a standard way~\citep{Engelke2015}. The exact $\Sigma$ is not required for our simulations

%In this case the extremal dependence can be studied and modeled for each group $I_g$ separately, and the mixing probabilities $p_g$ are then determined from $p_g c_g={\rm const}$ where $c_g$ is the analogue of $c$ in~\eqref{eq:means} for the group $I_g$.

%\figuresize{.6}
%\figurebox{20pc}{25pc}{}[fig1]
%\begin{figure}[h!]
%\begin{tabular}{ccc}
%\includegraphics[width=0.32\textwidth]{angles0.pdf}&
%\includegraphics[width=0.32\textwidth]{angles1.pdf}&
%\includegraphics[width=0.32\textwidth]{angles2.pdf}\\
%(a)&(b)&(c)
%\end{tabular}
%\caption{$\Sp^2$ with 100 points sampled from the exact angular distribution supported by $F_{\{1,2\}}\cup F_{\{3\}}$ in (a) and from its approximations in (b) and (c).}
%\label{fig:ex}
%\end{figure}

\section{Spherical clustering}
\label{sec:clustering}
\subsection{Spherical dissimilarities and Voronoi diagrams}
In this section we consider an arbitrary random variable $X\in\Sp^{d-1}$ not necessarily satisfying~\eqref{eq:means} and return to concomitant extremes in Section~\ref{sec:concom}.
Spherical clustering for a given integer $k\ges 1$ amounts to the following stochastic program:
\begin{equation}
\label{eq:clust}
\min_{x_1,\ldots,x_k\in\Sp^{d-1}} \e\left(\min_{i=1}^k\; c(x_i,X)\right)= 1-\max_{x_1,\ldots,x_k\in\Sp^{d-1}} \e\left(\max_{i=1}^k\; r(x_i^\top X)\right),
\end{equation}
where $c:\Sp^{d-1}\times\Sp^{d-1}\mapsto [0,1]$ is a continuous dissimilarity (cost) function. Here we assume that the dissimilarity function has the form $c(x,y)=1-r(x^\top y)$ for a strictly increasing continuous (reward) function $r:[0,1]\mapsto[0,1]$.
In words, the aim is to find $k$ \emph{centroids} $x_i\in\Sp^{d-1}$  such that the expected dissimilarity of $X$ and the closest centroid is minimal.
It is noted that the objective function is continuous in $(x_1,\ldots,x_k)$ and the set $\Sp^{d-1}$ is compact, which readily implies that the minimum must be attained.
Uniqueness up to a permutation, however, is not guaranteed.

There is a wide choice of popular angular dissimilarity functions used in a variety of contexts~\citep{palarea2012dealing}, but the above form has an appealing property
\[
c(x_1,y)\les c(x_2,y)\qquad\Leftrightarrow\qquad x_1^\top y\ges x_2^\top y\qquad \Leftrightarrow\qquad\norm{x_1-y}\les\norm{x_2-y},
\]
where $x_1,x_2,y\in\Sp^{d-1}$. Thus, for the given $k$ points $x_1,\ldots,x_k\in\Sp^{d-1}$ the positive part of the unit sphere will be partitioned into sets $A_i$ by means of hyperplanes in $\R^d$ passing through the origin:
\begin{equation}
\label{eq:partition}
A_i=\{y\in\Sp^{d-1}:\; x_i^\top y\ges x_j^\top y \;\forall j>i,\; x_i^\top y>x_j^\top y \;\forall j<i\},\quad i=1,\ldots,k.
\end{equation}
For concreteness, here the equality is broken in favour of the smallest index. Note also that these form the Voronoi diagram of $\Sp^{d-1}$ with respect to the Euclidean norm.

The most popular examples of such cost functions are the scaled Euclidean distance, the angular distance and cosine dissimilarity $c_1$, plus we additionally define $c_2$ corresponding to the quadratic reward function $r(u)=u^2$:
\begin{equation}
\label{eq:angle}
\sqrt{1-x^\top y},\qquad\sphericalangle(x,y)=\frac{2}{\pi}\arccos(x^\top y),\qquad c_p(x,y)= 1-(x^\top y)^p,\quad p=1,2.
\end{equation}
%Figure~\ref{fig:dissimilarities} illustrates these dissimilarities as function of the dot product $x^\top y$.
%\begin{figure}[h!]
%\centering
%\includegraphics[width=0.3\textwidth]{dissimilarities.pdf}
%\caption{Dissimilarities $c_1$ (black), $c_2$ (red), the scaled Euclidean norm (blue) and the angular distance (green) as functions of $x^\top y$.}
%\label{fig:dissimilarities}
%\end{figure}
The angular distance yields the angle between $x$ and $y$ as the fraction of the right angle $\pi/2$, and such a distance will be useful in numerical experiments below.
In clustering one of the most important properties is the simplicity and interpretability of the resulting procedure, and thus we will focus our attention on $c_p$ with $p=1,2$. In this case the above clustering problem~\eqref{eq:clust} reduces to maximizing the expected maximal reward $\e\big(\max_i (x^\top_iX)^p\big)$.
%\begin{equation}\label{eq:clust_reward}
%\max_{x_1,\ldots,x_k\in\Sp^{d-1}}\e \max_{i=1}^k (x^\top_iX)^p.\end{equation}
For $p=1$ we retrieve spherical $k$-means clustering~\citep{dhi2001}, which is intimately related to classical $k$-means. Finally, $p=2$ leads to what we call spherical $k$-pc clustering, which is described below.

\begin{remark}
Importantly, for $p=2$ we have $r'(0)=0$, which implies higher cost for two almost orthogonal vectors as compared to the $p=1$ case. This provides some heuristic support for using the $k$-pc method instead of $k$-means in identification of groups of concomitant extremes.
\end{remark}

\begin{remark}\label{rem:Hill}
While revising this work, we discovered that a similar algorithm had been considered in~\cite{Hill_et_al} under the name `perpendicular spherical $k$-means'.
The same cost function is used in this empirical work without observing a link to principal eigenvectors.
\end{remark}

\subsection{The optimal centroids and partitions}
Let us provide some basic results underlying the spherical clustering procedure for the linear and quadratic reward functions.
In the first case we will need the normalized means of $X$ restricted to the sets $A_i$:
\begin{equation}
\label{eq:xmean}
\overline x_i=\e(X\ind{X\in A_i})/\norm{\e(X\ind{X\in A_i})}\in\Sp^{d-1},
\end{equation}
where the partition $(A_1,\ldots,A_k)$ is given by~\eqref{eq:partition} for an arbitrary choice of $x_1,\ldots,x_k\in\Sp^{d-1}$.
In order to avoid division by~0 we take $\overline x_i=x_i$ when $\p(X\in A_i)=0$.
In the second case we define $d\times d$ matrices
\begin{equation}
\label{eq:sigmai}
\Sigma_i=\e(XX^\top\ind{X\in A_i}),\qquad i=1,\ldots,k.
\end{equation}
Note the difference between $\Sigma_i$ and $\Sigma_{\{i\}}$ defined in Section~\ref{sec:prelim}. These are non-negative symmetric matrices, which are also non-negative definite. We will be interested in the largest eigenvalue $\lambda_1(\Sigma_i)$ and the corresponding eigenvector $\hat x_i$ of $\Sigma_i$, assumed to have unit norm: $\norm{\hat x_i}=1$.
There may be more than one such eigenvector, but at least one of them must have non-negative entries by the Perron--Frobenius theorem, and we choose such. Hence we may assume that $\hat x_i\in\Sp^{d-1}$.
If the matrix $\Sigma_i$ is positive or at least irreducible, then such a vector is unique.

The first result forms the basis of the clustering procedure, see~\cite[Lem.\ 3.1]{dhi2001} for the sample version in the case of the spherical $k$-means algorithm.

\begin{proposition}
For $x_1,\ldots,x_k\in\Sp^{d-1}$ consider the partition $(A_i)$ in~\eqref{eq:partition}. Then
\begin{align*}
&\e\left(\max_{i=1}^k x^\top_iX\right)\les\e\left(\max_{i=1}^k \overline x^\top_iX\right),
&\e\left(\max_{i=1}^k (x^\top_iX)^2\right)\les\e\left(\max_{i=1}^k (\hat x^\top_iX)^2\right),
\end{align*}
where $\overline x_i$ are the normalized means defined in~\eqref{eq:xmean} and $\hat x_i$ are the normalized principal eigenvectors of the matrices $\Sigma_i$ defined in~\eqref{eq:sigmai}.
\end{proposition}

\begin{proof}
By definition of $A_i$ and $\overline x_i$ we readily obtain
\begin{equation*}
\e\left(\max_{i=1}^k x^\top_iX\right)=\sum_{i=1}^k x^\top_i\e(X\ind{X\in A_i})\les \sum_{i=1}^k \overline x_i^\top\e(X\ind{X\in A_i})\les\e\left(\max_{i=1}^k \overline x^\top_iX\right),
\end{equation*}
where the last inequality follows from the fact that only one term among $X\ind{X\in A_i}$ is non-zero.

By definition of $\Sigma_i$ we have
\begin{equation}
\label{eq:eig_proof}
\e\left(\max_{i=1}^k (x^\top_iX)^2\right)=\sum_{i=1}^k \e(x^\top_iX\ind{X\in A_i})^2= \sum_{i=1}^k x^\top_i\Sigma_ix_i\les\sum_{i=1}^k \hat x^\top_i\Sigma_i\hat x_i= \sum_{i=1}^k \lambda_1(\Sigma_i),
\end{equation}
where we use the standard fact that the quadratic form is maximized under the constraint $\norm{x}=1$ by the principal eigenvector and the maximal value is given by the corresponding eigenvalue, see, e.g.,~\cite[Thm.\ 1]{overton1992sum}. It is left to observe that
\begin{equation}
\label{eq:eig_proof2}
\sum_{i=1}^k \hat x^\top_i\Sigma_i\hat x_i=\sum_{i=1}^k \e(\hat x_i^\top X\ind{X\in A_i})^2\les \e\left(\max_{i=1}^k (\hat x^\top_iX)^2\right),
\end{equation}
since only one term among $X\ind{X\in A_i}$ is non-zero.
\end{proof}

The above result suggests an iterative procedure for finding the optimal centroids where the updates $\overline x_i$, respectively $\hat x_i$, of the current centroids $x_i$ are used. 
This procedure will normally converge to a local maximum for the linear, respectively quadratic, reward function.
By using a number of different starting centroids we may hope to discover the global maximum and thus solve~\eqref{eq:clust} for the dissimilarity functions $c_1$ and $c_2$.
This is exactly the spherical $k$-means procedure when $c_1$ and hence the means $\overline x_i$ are used.

Importantly, instead of centroids we may optimize over partitions of~$\Sp^{d-1}$:

\begin{corollary}[Duality]
\label{cor:dual}
Let $\mathcal P_k$ be the set of partitions of $\Sp^{d-1}$ into $k$ Borel sets. Then
\begin{align*}
\max_{x_1,\ldots,x_k\in\Sp^{d-1}} \e\left(\max_{i=1}^k x_i^\top X\right)&=\max_{(A_1,\ldots A_k)\in\mathcal P_k} \sum_{i=1}^k \norm{\e(X\ind{X\in A_i})},\\
\max_{x_1,\ldots,x_k\in\Sp^{d-1}} \e\left(\max_{i=1}^k (x_i^\top X)^2\right)&=\max_{(A_1,\ldots A_k)\in\mathcal P_k} \sum_{i=1}^k \lambda_1(\Sigma_i),
\end{align*}
where $\Sigma_i$ are defined in~\eqref{eq:sigmai}. Moreover,
\begin{itemize}
\item every optimizer $(x_1,\ldots,x_k)$ yields an optimal partition $(A_1,\ldots,A_k)$ via~\eqref{eq:partition};

\item every optimal partition $(A_1,\ldots,A_k)$ yields an optimal $(x_1,\ldots,x_k)$  given by $x_i=\overline x_i$ in the first case and $x_i=\hat x_i$ in the second.
\end{itemize}
%and the optimal $\tilde A_i$ yield an optimizer of~\eqref{eq:cosine_clust} given by $\tilde c_i=\e(X\ind{X\in \tilde A_i})/\norm{\e(X\ind{X\in \tilde A_i}}$.
\end{corollary}

\begin{proof}
The proof of the two cases is analogous, and we consider only the second case.
Let $x_1,\ldots,x_k\in\Sp^{d-1}$ be an optimal solution to the left-hand side problem, which exists.
By the maximality and~\eqref{eq:eig_proof} we find that the optimal value is $\sum_{i=1}^k \hat x_i^\top\Sigma_i\hat x_i=\sum_{i=1}^k\lambda_1(\Sigma_i)$.
So the supremum over partitions is no smaller.
Take an arbitrary partition $(A_1,\ldots,A_k)\in\mathcal P_k$ with the respective $\hat x_i$
and observe that the associated value cannot exceed the maximum over the centroids, see~\eqref{eq:eig_proof2}.
Hence the supremum over the partitions is achieved and the optimal values coincide.
\end{proof}

\begin{remark}
\label{rem:sigma}
Matrix $\Sigma$ has been used by~\cite{coo2019} and~\cite{dre2019} in their principal component analysis of extremal dependence.
Finding the best direction for $X\in\Sp^{d-1}$ corresponds to minimizing the expected $c_2$ dissimilarity:
\[
\min_{x:\norm{x}=1} \e\left(\norm{X-(x^\top X)x}^2\right)=1-\max_{x:\norm{x}=1} \e(x^\top X)^2= \min_{x:\norm{x}=1} \e\left(c_2(x,X)\right).
\]
In the trivial case of $k=1$ the corresponding centroid $\hat x_1=\lambda_1(\Sigma)$ is the main PCA direction.
No such links seem to exist for $k\ges 2$.
%This provides a link between our clustering approach and principal component analysis for a distribution on~$\Sp^{d-1}$.
%It must be noted that using additional directions in PCA may easily result in an approximation of $X$ with some entries being negative.
\end{remark}

\subsection{Spherical $k$-principal-components algorithm}
Here we state the algorithm for a discrete distribution putting mass $1/n$ at points $\theta_1,\ldots,\theta_n\in\Sp^{d-1}$, not necessarily distinct. Clearly, this setting includes the empirical law. Only a single iteration is described, since the rest is standard.  R~code is available at \url{https://github.com/jev-ivanovs/spherical_clustering}.

\begin{algo}
Spherical $k$-principal-components clustering -- a single iteration.
\begin{tabbing}
\qquad\enspace\underline{Input}: the sample $\theta_1,\ldots,\theta_n\in\Sp^{d-1}$ and current centroids $x_1,\ldots,x_k\in\Sp^{d-1}$.\\
\qquad\enspace Compute $n\times k$ matrix of dot products $M=(\theta_1,\ldots,\theta_n)^\top(x_1,\ldots,x_k)$\\
\qquad\enspace Let $v$ be the mean of row-wise maxima of $M$\\
\qquad\enspace For each row of $M$ find the index of the (first) maximal value and store them in~$g$\\
\qquad\enspace For $i=1$ to $i=k$\\
	\qquad\qquad Calculate $\Sigma_i=\frac{1}{n}\sum_{u=1}^n (\theta_u\theta_u^\top\ind{g_u=i})$\\
	\qquad\qquad Find the principal eigenvector $\hat x_i\in\Sp^{d-1}$ of $\Sigma_i$\\
\qquad\enspace\underline{Output}: new centroids $\hat x_1,\ldots,\hat x_k\in\Sp^{d-1}$ and the old value $v$.
\end{tabbing}
\end{algo}

It is easy to see that the running time complexity of this algorithm is $O(k(nd+f(d)))$, where $f(d)$ is the complexity of finding the principal eigenvector of $d\times d$ non-negative symmetric matrix, see~\cite{wang2018efficient} for some basic algorithms and recent developments in this field. Assuming that the spectral gap is bounded away from 0 we may take $f(d)=d^2\log d$, making $k$-means and $k$-pc comparable in the common situation when $d\log d\les n$.

\section{Spherical clustering and concomitant extremes}
\label{sec:concom}
\subsection{Problem formulation}
The main goal of this work is to provide some theoretical results supporting clustering for identification of concomitant extremes.
The dissimilarities $c_1$ and $c_2$ are continuous and hence the consistency result~\cite[Prop.\ 3.3]{Janssen_Wan} readily applies to both types of spherical clustering described above.
In words, for high enough threshold $t$ yielding sufficiently many large observations we may expect that clustering of the approximate angles will result in centroids $x_1,\ldots,x_k$ close to the true centroids of the angular distribution, given the latter ones are unique up to a permutation.
Thus, we may focus on clustering of the exact angular distribution, that is, the distribution of $X$. The balancing condition~\eqref{eq:means} will play a crucial role on this way. Some results are still true without this condition and so we stress when it is indeed required.

The basic test scenario is stated in Assumption~\ref{assumption}, and it gives rise to the following question:
\[
\text{Will spherical clustering of $X$ under Assumption~\ref{assumption} produce one centroid in each face?}
\]
We will see that this is not always the case, and our goal is to identify simple interpretable conditions implying such a result. It turns out that spherical $k$-pc is preferable to spherical $k$-means in this setting, since it is more robust and also allows for a substantial theory.

Figure~\ref{fig:ex_clust} illustrates spherical $k$-means clustering in the simple case of $d=3$ and different approximation levels of the angular distribution. Spherical $k$-pc clustering results in exactly the same assignment of all points (blue/green) to the clusters.
The centroids for the two methods are close to each other and also close to the respective faces, and in the case of sampling from the exact law they are, in fact, on the faces $F_{\{1,2\}}$ and $F_{\{3\}}$.

\subsection{Fundamental observations}
\label{sec:counterex}
A partial answer to our problem is given by the following characterization result, which readily follows from the clustering duality in Corollary~\ref{cor:dual}.
The difficult part is in establishing simple sufficient conditions, which we address in Theorem~\ref{thm:means} and Theorem~\ref{thm:eigs}.
Recall the definition of the submodels $X_I$ in Section~\ref{sec:prelim} and that $\mathcal P_k$ is the set of all partitions of $\Sp^{d-1}$ into $k$ Borel sets (or sets of the form~\eqref{eq:partition}).

\begin{proposition}
\label{prop:main}
Under Assumption~\ref{assumption} the following is true.
\begin{itemize}
\item $k$-means: There exist optimal centroids $x_1\in F_{I_1},\ldots,x_k\in F_{I_k}$ iff
\begin{equation}
\label{eq:cond_means}
\sum_{I=I_1,\ldots,I_k} p_I\norm{\e (X_I)}=\max_{(A_1,\ldots,A_k)\in\mathcal P_k} \sum_{i=1}^k \norm{\e(X\ind{X\in A_i})},
\end{equation}
where the left-hand side is $\mu(\sqrt{|I_1|}+\cdots+\sqrt{|I_k|})$ assuming~\eqref{eq:means}.

\item $k$-pc: There exist optimal centroids $x_1\in F_{I_1},\ldots,x_k\in F_{I_k}$ iff
\begin{equation}
\label{eq:cond_eigs}
\sum_{I=I_1,\ldots,I_k} p_I\lambda_1(\Sigma_I)=\max_{(A_1,\ldots,A_k)\in\mathcal P_k} \sum_{i=1}^k \lambda_1\left(\e(XX^\top\ind{X\in A_i})\right).
\end{equation}
\end{itemize}
\end{proposition}

\begin{proof}
We focus on $k$-means, since the other result is analogous. Note that the left-hand side of~\eqref{eq:cond_means} is just $\sum_{i=1}^k \norm{\e(X\ind{X\in F_{I_i}})}$ and so by Corollary~\ref{cor:dual}
\[
x_i=\e(X\ind{X\in F_{I_i}})/\norm{\e(X\ind{X\in F_{I_i}})}
\]
yields an optimal set of centroids, since the given disjoint sets $F_{I_i}\cap\Sp^{d-1}$ can be extended to a partition of $\Sp^{d-1}$. These indeed satisfy $x_i\in F_{I_i}$.

Next, suppose that some $x_i\in F_{I_i}$ are optimal, and let $A_i$ be the sets in~\eqref{eq:partition}, which then yield the maximum. The faces are mutually orthogonal and so every vector $y\in A_i$ in the support of $X$ also belongs to $F_{I_i}$, unless $x_i^\top y=0$ for all $i$. The latter vectors can be reshuffled between the associated clusters without changing the cost. So we may restrict $A_i$ to $F_{I_i}$ and the first statement follows. Moreover, assuming~\eqref{eq:means} we have $\e(X_I)=\mu_I(1,\ldots,1)^\top$ and so its norm is $\mu_I\sqrt{|I|}$.

In the case of $k$-pc we also need to observe that the principal eigenvector of $\e(XX^\top\ind{X\in F_{I_i}})$ indeed belongs to~$F_{I_i}$, and the corresponding eigenvalue is $p_{I_i}\lambda_1(\Sigma_{I_i})>0$.
\end{proof}

A simple critical test of any approach is given by the angular distribution corresponding to the asymptotic independence, see case (i) in Section~\ref{sec:prelim}. In this case any partition of indices yields faces satisfying Assumption~\ref{assumption}. We take $k=2$ for simplicity and partition the index set into $I_1=\{1,\ldots,d_1\}$ and $I_2=\{d_1+1,\ldots,d\}$ for some $1\les d_1\les d-1$.
For $k$-means the necessary and sufficient condition~\eqref{eq:cond_means} reads
\[
\sqrt{d_1}+\sqrt{d-d_1}=\max_{\ell=0,\ldots,d} \{\sqrt\ell+\sqrt{d-\ell}\},
\]
where $\ell$ and $d-\ell$ correspond to the alternative partition.
But the right-hand side is maximized at $\ell=d/2$ when $d$ is even and $\ell=(d\pm 1)/2$ when $d$ is odd, and so we must choose index sets of essentially equal size to make faces identifiable using $k$-means, see also  Theorem~\ref{thm:means} below.
In the $k$-pc case we note that $\lambda_1(\e(XX^\top\ind{X\in A_i}))=1/d$ whenever $A_i$ contains at least one standard basis vector.
Thus, the necessary and sufficient condition~\eqref{eq:cond_eigs} always holds for such~$X$.

In the above case the partitioning is arbitrary, but one can always construct another angle $X'$ satisfying the moment constraints and arbitrarily close to $X$ such that a given partition becomes the only correct one (the respective faces support the law of $X'$).
This provides a class of examples where $k$-means clustering fails to identify the supporting faces, because the dissimilarity is continuous, see also~\cite{Janssen_Wan}. Importantly, $k$-pc does not readily fail in this case.

\subsection{Spherical $k$-means in the size-balanced case}
The spherical $k$-means procedure is guaranteed to identify the correct faces only when their dimensions satisfy a certain strict condition.
%For $k=2$ it states that the dimensions of the two faces must be maximally close to each other, which is a very restrictive assumption for a large dimension~$d$.
If this assumption is violated, it is possible to construct a model where $k$-means fails, see the above discussion. % as explained above. In fact, asymptotic independence may be heavily relaxed in such counterexamples.
%It is sufficient to have an alternative partition into faces with a larger value of $\sqrt{|I_1|}+\cdots+\sqrt{|I_k|}$, or rather a model relatively close to that, as discussed above. %This observation readily follows from Proposition~\ref{prop:main}.

\begin{theorem}
\label{thm:means}
Suppose Assumption~\ref{assumption} and~\eqref{eq:means} hold.
If $d_1=|I_1|,\ldots,d_k=|I_k|$ yield the maximum in
\[
\max_{\substack{d_1,\ldots,d_k\in\mathbb N_0\\d_1+\cdots+d_k=d}} \sqrt{d_1}+\cdots+\sqrt{d_k},
\]
then there exist optimal centroids $(x_1,\ldots,x_k)$ of spherical $k$-means clustering with $x_i\in F_{I_i}$ for all $i$.
For $k=2$ this condition reads $\big||I_1|-|I_2|\big|\les 1$.
\end{theorem}

\begin{proof}
In view of~\eqref{eq:cond_means} it is only required to show  that
\[
\mu(\sqrt{d_1}+\cdots+\sqrt{d_k})\ges\sum_{i=1}^k \norm{\e(X\ind{X\in A_i})}
\]
for any partition $(A_1,\ldots,A_k)\in\mathcal P_k$. We denote the right-hand side by $\sum_{i=1}^k \norm{\nu_i}$ and note that its maximum over $\nu_i\in\R^d_+$ subject to the constraint $\sum \nu_i=\e(X)=\mu(1,\ldots,1)^\top$ is attained under the assumption $\nu_i\perp\nu_j$ for all $i\neq j$, see Lemma~\ref{lem:triangle}.
Letting $\delta_i$ be the number of non-zero entries in~$\nu_i$ we get an upper bound on the sum of norms:
\[
\mu(\sqrt{\delta_1}+\cdots+\sqrt{\delta_k}),
\]
and the first statement follows.
The case $k=2$ has been discussed above and here the optimal integers $d_1,d_2$ are such that $|d_1-d_2|\les 1$. The proof is complete.
\end{proof}

\section{Spherical $k$-principal-components clustering and concomitant extremes}
\subsection{The main result}
The eigenvalues of the matrices $\Sigma_I$ corresponding to the submodels, see Section~\ref{sec:prelim}, will play a crucial role in the following. When discussing some basic properties of these matrices we write $\Sigma$ and assume that $d\ges 1$.
Also, let $\lambda_i(M)$ denote the $j$th largest eigenvalue of a symmetric matrix~$M$, tacitly assuming $\lambda_i(M)=0$ when $j$ exceeds the order of~$M$.
In the case (i) of asymptotic independence $\Sigma$ is a diagonal matrix with $1/d$ on the diagonal, so that $\lambda_i(\Sigma)=1/d=\mu$. In the case (ii) of complete dependence $\Sigma$ is a matrix with constant elements $1/d$ yielding $\lambda_1(\Sigma)=1,\lambda_2(\Sigma)=0$.
The following lower bound on the largest eigenvalue is true in general.

\begin{lemma}
\label{lem:eigenvalues2}
Consider $\Sigma=\e(XX^\top)$, where $X\in\Sp^{d-1}$ with $d\ges 1$ satisfies~\eqref{eq:means}. Then
\[
\lambda_1(\Sigma)\ges\mu,
\]
and the equality implies asymptotic independence.
\end{lemma}

\begin{proof}
Consider the decomposition $\Sigma=\Sigma_0+\mu^2\boldsymbol{1}\boldsymbol{1}^\top$, where $\Sigma_0$ is the respective covariance matrix and $\boldsymbol{1}$ is a vector of ones.
All three matrices are non-negative definite, and the eigenvalues of the latter are $d\mu^2,0,\ldots,0$.
Next, we use the standard inequality: $\lambda_1(\Sigma)\ges \lambda_1(\mu^2\boldsymbol{1}\boldsymbol{1}^\top)=d\mu^2$, which follows from the interpretation of $\lambda_1$ as the maximum over quadratic forms, for example.
Apply Lemma~\ref{lem:mu} to conclude.
\end{proof}

We are now ready to state our main result providing some basic theory supporting the use of clustering in detection of groups of concomitant extremes.

\begin{theorem}
\label{thm:eigs}
Suppose Assumption~\ref{assumption} holds. If
\begin{equation}
\label{eq:minmax}
\min_{I=I_1,\ldots,I_k} p_I\lambda_1(\Sigma_I)\ges\max_{I=I_1,\ldots,I_k} p_I\lambda_2(\Sigma_I),
\end{equation}
then there exist optimal centroids $(x_1,\ldots,x_k)$ of spherical $k$-principal-components clustering such that $x_i\in F_{I_i}$ for all $i$.
Moreover, this condition is satisfied when~\eqref{eq:means} holds and
\begin{equation}
\label{eq:lambda2}
\lambda_2(\Sigma_I)\les\mu_I\qquad I=I_1,\ldots,I_k.
\end{equation}
\end{theorem}

Importantly, this sufficient condition asserts that the second principal direction for any face provides a smaller cost reduction than the principal direction of any other face taking the respective probability weights into account. That is, the problem is balanced in this sense. Moreover, the above condition is implied when $\mu_I$ separates the first two eigenvalues of $\Sigma_I$ for all submodels $I$, see Lemma~\ref{lem:eigenvalues2}. This is true, for example, for certain symmetric models, see Lemma~\ref{lem:symmetric} below, irrespective of the dimension.
% that clustering within each face is rather good with respect to the mass of this face.

\begin{proof}[Proof of Theorem~\ref{thm:eigs}]
Let $D_i=p_{I_i}\Sigma_{I_i}$ be the diagonal blocks of $\Sigma$.
In view of~\eqref{eq:cond_eigs}, it is only required to show for any partition $(A_1,\ldots, A_k)\in\mathcal P_k$ that
\[
\sum_{i=1}^k \lambda_1(D_i)\ges\sum_{i=1}^k \lambda_1(\Sigma_i),
\]
where $\Sigma_i=\e(XX^\top\ind{X\in A_i})$.
%Here we also used an obvious fact that $\lambda_1(XX^\top\ind{X\in F_i})=\lambda_1(D_i)$.
Lemma~\ref{lemma:eigenvalues} in the Appendix is crucial here, and it shows that the right-hand side is upper bounded by $\sum_{i=1}^k \lambda_i(\Sigma)$. Thus it is left to show that
$\sum_{i=1}^k \lambda_1(D_i)\ges\sum_{i=1}^k \lambda_i(\Sigma)$, which is indeed equivalent to the stated assumption, because of the block diagonal structure of $\Sigma=\diag(D_1,\ldots,D_k)$.

Now suppose that~\eqref{eq:means} and~\eqref{eq:lambda2} are in place. Then $\max_I \{p_I\lambda_2(\Sigma_I)\}\les\mu\les\min_I \{p_I\lambda_1(\Sigma_I)\}$, where we apply Lemma~\ref{lem:eigenvalues2} to each $\Sigma_I$.
\end{proof}

\subsection{Simple sufficient conditions}
Let us discuss the condition in~\eqref{eq:lambda2} for a fixed~$I$. Firstly, it is always true for a one- or two-dimensional face, where in the former case we tacitly assume that $\lambda_2=0$.

\begin{lemma}
Assume~\eqref{eq:means} and $d=2$. Then $\lambda_2(\Sigma)\les\mu$.
\end{lemma}

\begin{proof}
We have $\lambda_1+\lambda_2=\mathrm{tr}\Sigma=1$, and $\lambda_1\ges\mu$ by Lemma~\ref{lemma:eigenvalues}. By Lemma~\ref{lem:mu}, $\mu\ges 1/2$, and so $\lambda_2=1-\lambda_1\les 1-\mu\les\mu$, which completes the proof.
\end{proof}

Secondly, certain symmetries imply this condition as, for example, invariance of the second moments under permutations.

\begin{lemma}
\label{lem:symmetric}
Assume~\eqref{eq:means} and $\e(X_iX_j)=\e(X_{\pi(i)}X_{\pi(j)})$ for all $i,j$ and all permutations $\pi$. Then $\lambda_2(\Sigma)\les 1/d\les\mu$.
\end{lemma}

\begin{proof}
Let $c=\e(X_1X_2)\ges 0$ be the common cross-moment. Observe from $\norm{X}=1$ that $\e(X_1^2)=1/d$, so that $\Sigma$ is a matrix with diagonal elements $1/d$ and off-diagonal~$c$.
It is not difficult to check that $\lambda_1=1/d+(d-1)c$ and the other eigenvalues are all equal to $1/d-c$. The proof is complete in view of Lemma~\ref{lem:mu}.
\end{proof}

\subsection{Counterexamples}
Basic counterexamples to successful identification of faces are obtained by considering more groups of (almost) concomitant extremes than there are clusters.
Assume that the index set can be partitioned into 3 sets $I_1,I_2,I_3$, but we use $k=2$ with the partition $I_1$ and $I_2\cup I_3$.
Write $\Sigma=\diag(D_1,D_2,D_3)$ and suppose that the leading eigenvalues satisfy
\begin{equation}
\label{eq:ass_counter}
\lambda_1(D_1)<\lambda_1(D_2)\les\lambda_1(D_3).
\end{equation}
Now we readily find that an alternative partition $I_1\cup I_2$ and $I_3$ produces a strictly larger value:
\[
\lambda_1(D_1)+\lambda_1(\diag(D_2,D_3))<\lambda_1(\diag(D_1,D_2))+\lambda_1(D_3),
\]
showing that $k$-pc will fail to identify the faces $F_{I_1}$ and $F_{I_2\cup I_3}$, see Proposition~\ref{prop:main} and~\eqref{eq:Sigma}.
Of course, the failure is due to the wrong grouping of the three faces, but by continuity of the dissimilarity function we can perturb the model by introducing some small mass on $F_{I_2\cup I_3}$ without drastically changing the location of the optimal centroids, but making the partition $I_1$ and $I_2\cup I_3$ the only possibility.

An example with the property~\eqref{eq:ass_counter} can be produced by taking $|I_1|=1$ and $|I_2|,|I_3|>1$ with the latter two submodels being asymptotically dependent, so that $\lambda_1(D_1)=\mu<\lambda_1(D_2)\wedge\lambda_1(D_3)$ according to Lemma~\ref{lem:eigenvalues2}. Examples with $|I_1|>1$ can be obtained by imposing higher dependence in the subsets $I_2$ and~$I_3$.

\section{Numerical experiments}
\label{sec:numerics}
\subsection{The simulation framework}
Firstly, we consider a $d=100$ dimensional model satisfying Assumption~\ref{assumption} with $k=2$ faces.
The $k$-means and $k$-pc algorithms are compared by attempting to associate each pair of centroids with the two faces and computing certain scores.
Secondly, we consider a more practical situation where the groups are not disjoint and the number of groups $k$ must be inferred from data.
Finally, we apply both clustering algorithms to two real world datasets.

Here we briefly describe our basic model used in simulations.
The random vector $Y$ is taken to have a $d$-variate max-stable H\"usler--Reiss distribution introduced in~\cite{Husler1989} and further studied in~\cite{Engelke2015,eng2018}. This popular family has a good control of pairwise extremal dependencies which are encoded into the associated variogram matrix $\Gamma$ parametrizing the distribution.
The variogram satisfying Assumption~\ref{assumption} with $k=2$ faces of dimensions $d_1$ and $d_2$ is produced randomly in a certain way detailed in the Supplementary Material.
The pairwise tail dependence coefficients are likely to be very small in our parameter generation procedure.
This may lead to a subdivision of a group of concomitant extremes into almost independent subgroups making the detection problem harder, see Section~\ref{sec:concom}.
%Note that $\chi_{ij}$ has an explicit expression in terms of $\Gamma_{ij}$ and need not be estimated from approximate angles.

In each experiment we sample $10^4$ i.i.d.\ realizations of~$Y$ using the R package~\cite{graphicalExtremes}. We take $10\%$ of these vectors having the largest Euclidean norm and thus form $10^3$ approximate realizations of the angle~$X$. An example of the estimated matrix $(\chi_{ij})$ appears in Figure~\ref{fig:nonorthogonal_model} (left). The mean is $0.21$  for asymptotically dependent $i\neq j$ and $0.1$ for the pairs with asymptotic independence.

One way to obtain faces from the centroids $x\in\Sp^{d-1}$ is to use a simple truncation procedure: $I=\{i:x_i>\delta\}$ for a chosen level $\delta>0$. Another way is to find the index set $I$ with the minimal cardinality and such that the angle with the respective face is sufficiently small: $\sphericalangle(x,F_I)<\v$, where
\begin{equation}
\label{eq:angleF}
\sphericalangle(x,F_I)=\min\{\sphericalangle(x,y):y\in F_I\cap\Sp^{d-1}\}= \frac{2}{\pi}\arccos\big[\big(\sum_{i\in I} x^2_i\big)^{1/2}\big]\in [0,1],
\end{equation}
see also~\eqref{eq:angle}.
The last equality follows from the fact that the minimizer $y$ is given by the normalized vector $(x_i\ind{i\in I})_{i=1}^d$. Finding such $I$ is simple, since we only need to choose enough indices corresponding to the largest~$x_i$'s. Note that by doing so we always pick a face yielding the smallest angle with~$x$ among all faces of the same dimension. This latter approach is more consistent with the spherical clustering paradigm, and so we use it throughout our experiments. In fact, any dissimilarity in Section~\ref{sec:clustering} would produce the same results up to changing the threshold appropriately.

\subsection{Comparison of $k$-means and $k$-pc}
\begin{figure}[!ht]
\centering
\includegraphics[width=0.3\textwidth]{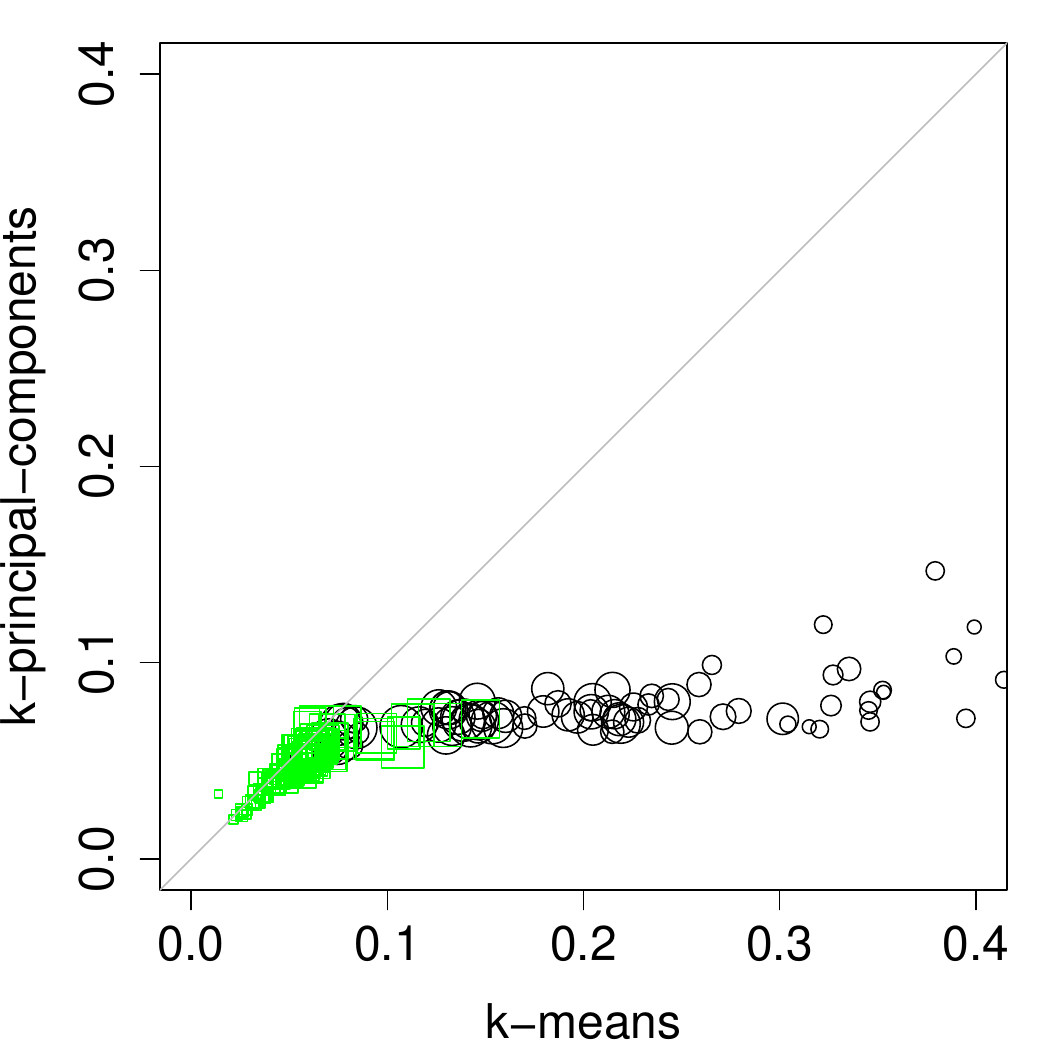}\qquad
\includegraphics[width=0.3\textwidth]{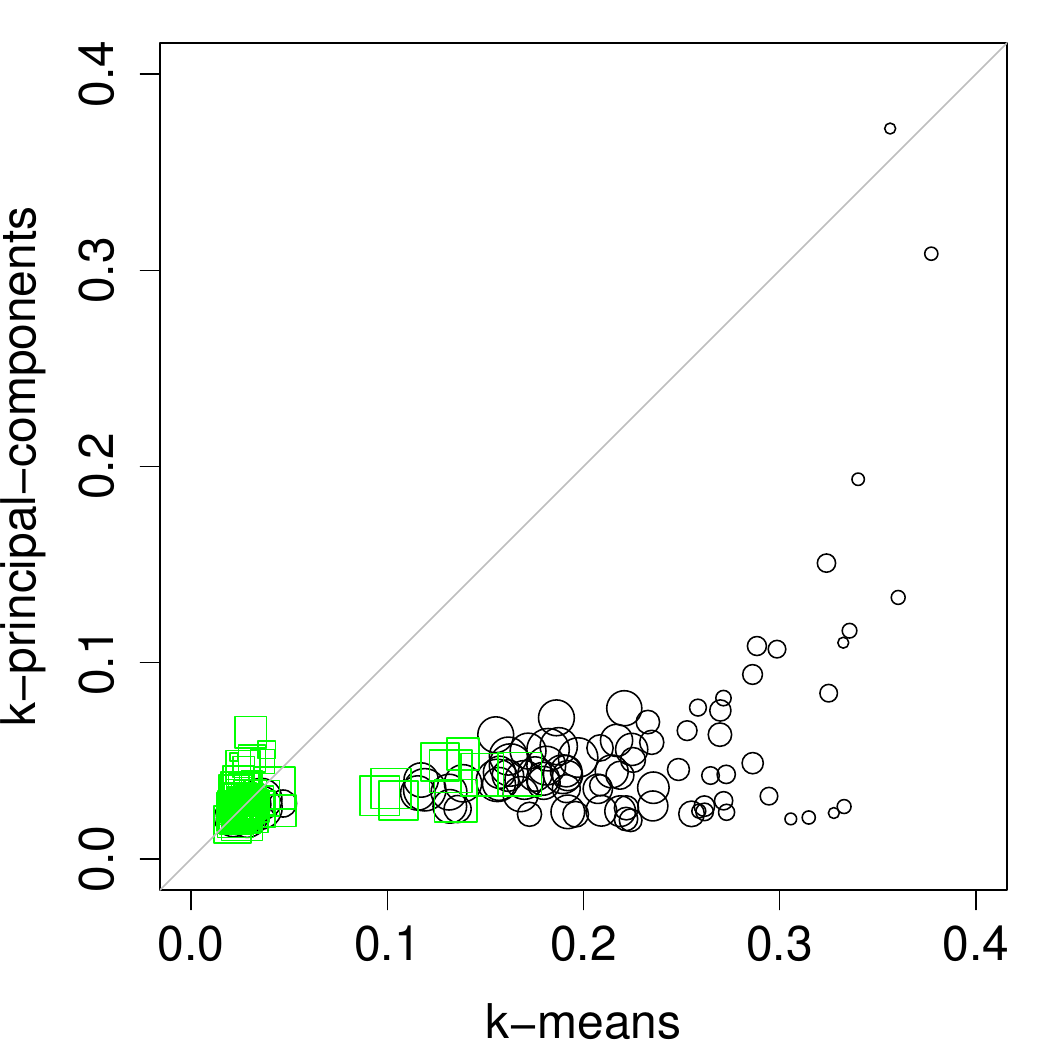}
\caption{Left: angles to the corresponding faces. Right: maximal entries over $I^c$. Black circles correspond to the first face, green squares to the second, and scale to~$d_1$.}
\label{fig:compar}
\end{figure}

We replicate the following procedure $100$ times: we choose $d_1\in\{1,\ldots,50\}$ uniformly at random and generate a variogram~$\Gamma$ corresponding to the partition $I_1=\{1,\ldots,d_1\}$ and $I_2=\{d_1+1,\ldots,100\}$. Then we produce $10^3$ approximate realizations of the angle~$X$, apply both the $k$-means and $k$-pc algorithms using $k=2$ and 100 random restarts in each, and associate each pair of centroids with the two faces.
For each centroid $x$ we calculate the angle to the corresponding face $\sphericalangle(x,F_I)$ defined in~\eqref{eq:angleF} and the maximal entry $\max\{x_i:i\in I^c\}$ over the indices defining the other face. That is, a small number indicates that the centroid is indeed close to its face, see Figure~\ref{fig:compar} where 19 and 8 black points are outside the plot range, respectively.

We call it an error when a centroid yields the angle and the maximal entry both exceeding~0.1.
Out of 200 trials (faces) there are approximately $46\%$ errors for $k$-means and only $8\%$ for $k$-pc, and all of the latter are also the errors for the $k$-means procedure.
%91 errors for $k$-means and only 15 for $k$-principal components, and all of the latter are also the errors in $k$-means.
By increasing the threshold to 0.2 we get $26\%$ and $5\%$ of errors, and again all the errors in the latter are also errors in the former. Furthermore, these correspond to the first face and occur when it is relatively small.
We also check that in all these cases permuting the assignment to faces does not resolve the issue.

In the above regime of weak asymptotic dependence the $k$-pc method outperforms $k$-means, which is also expected in view of the above developed theory and accompanying intuition.
In the case of increased dependence within the faces both clustering methods perform extremely well, making identification of the corresponding faces an obvious task in this simplistic setting.
%Note that even in this simplistic setting the direct thresholding approach produces a large number of faces with a few observations, and so further grouping is required, see Figure~\ref{fig:elbow} (left) for a similar plot.
%But even in this simple-looking setting the direct approach based on data thresholding fails to provide any useful information, even upon further aggregation and threshold tuning.

\subsection{Non-orthogonal faces and the choice of $k$}
Recall that Assumption~\ref{assumption} is required for our theoretical guarantees alone.
Here we consider a typical scenario where the groups of concomitant extremes have common elements and the number of groups $k$ must be inferred from data.
We mix two datasets: half of the observations come from the above (randomly sampled) H\"usler--Reiss model with two groups $\{1,\ldots, 40\},\{61,\ldots,100\}$ and half from an analogous model with two groups $\{21,\ldots,60\},\{1,\ldots,20,61,\ldots,100\}$.
These are the underlying four faces to be identified.
The estimated matrix of tail dependence coefficients $(\chi_{ij})$ is presented in Figure~\ref{fig:nonorthogonal_model} (left).
This image may give a false feeling that identification of groups is easy, and so we reorder the indices according to the standard heatmap routine in R based on hierarchical clustering, see Figure~\ref{fig:nonorthogonal_model} (centre).
Now the groups cannot be easily seen with a naked eye.

\begin{figure}[!ht]
\centering
\includegraphics[clip, trim=0 0 20 0, height=0.3\textwidth]{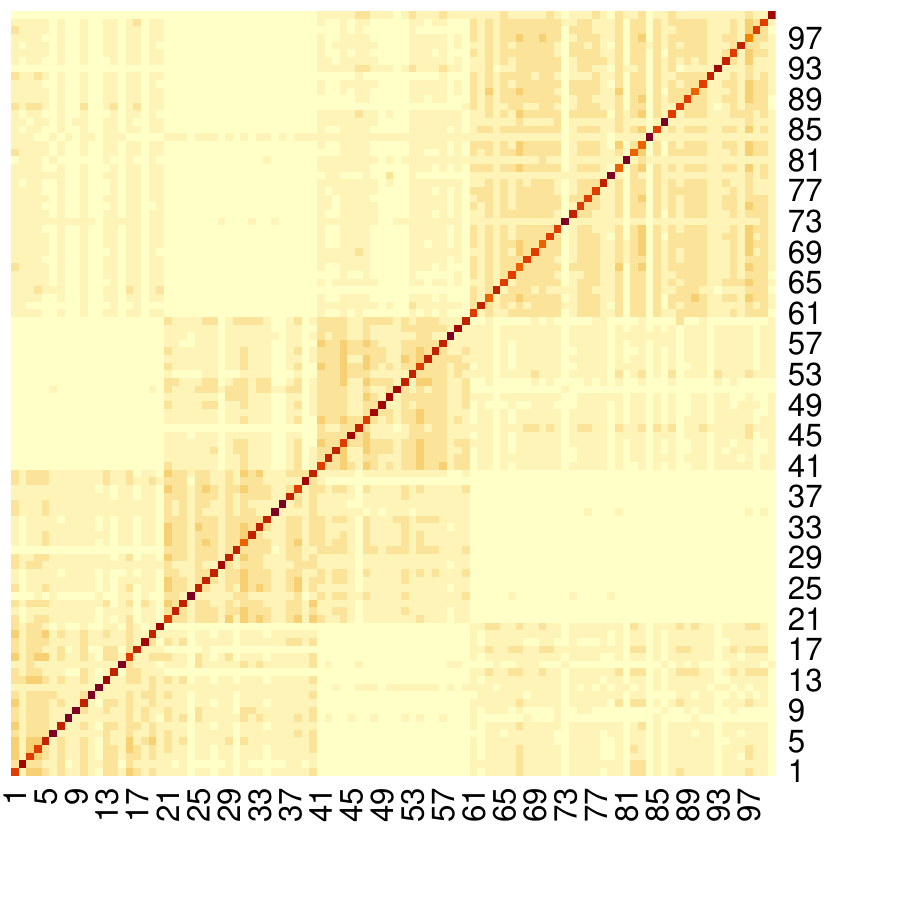}
\includegraphics[clip, trim=0 10 0 85, height=0.3\textwidth]{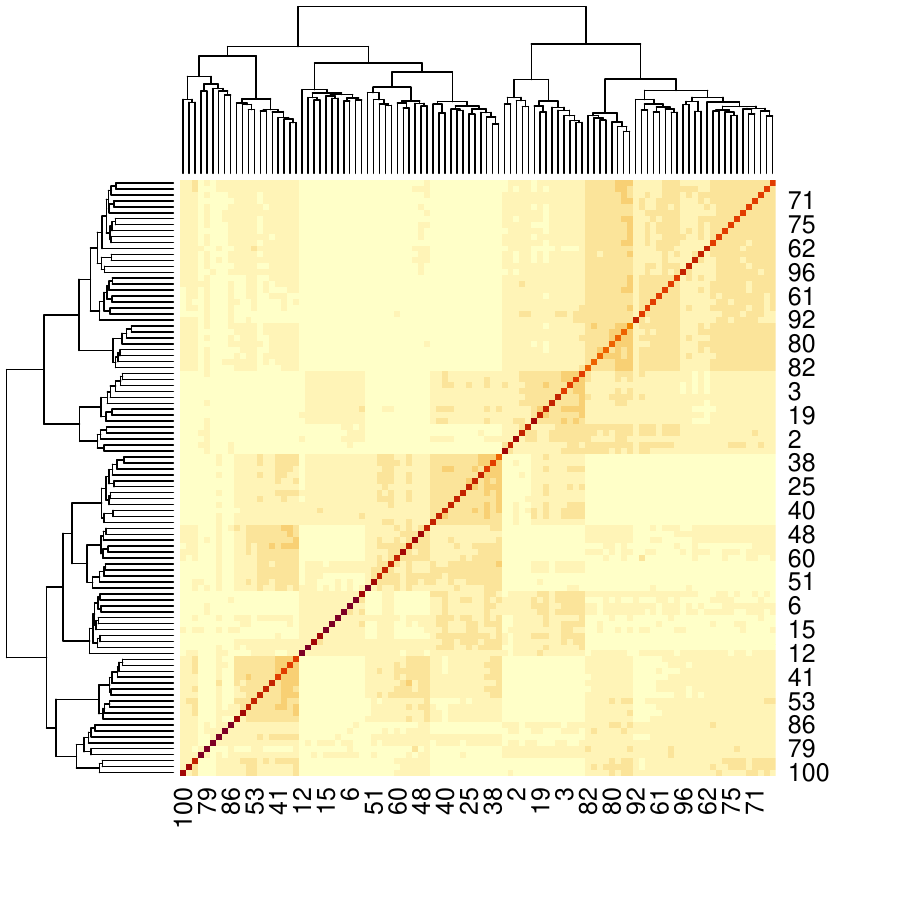}
\includegraphics[clip, trim=30 20 50 70, width=0.3\textwidth]{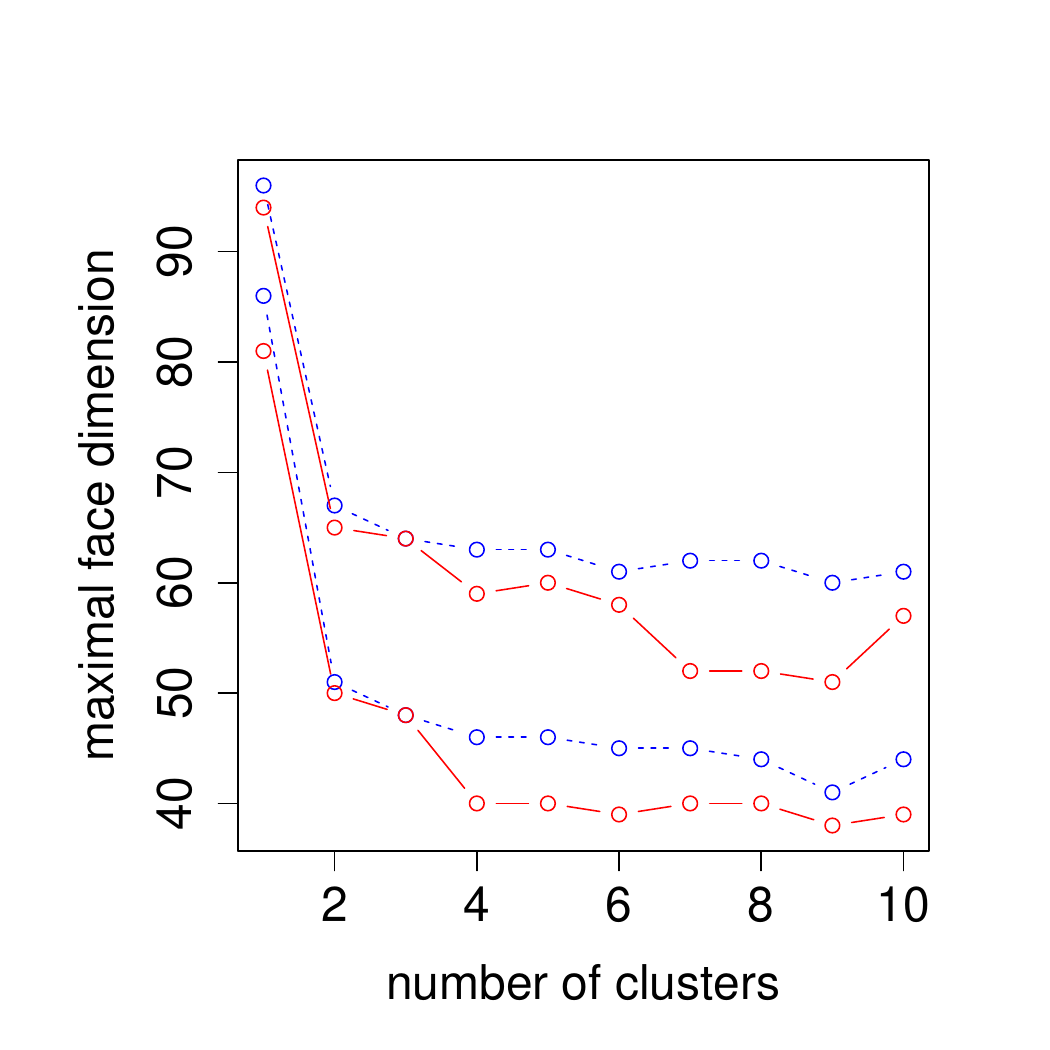}
\caption{Non-orthogonal faces: estimated $(\chi_{ij})$, its heatmap, and maximal face dimension for $\v\in\{1/10,1/5\}$  with $k$-pc in solid red. }
\label{fig:nonorthogonal_model}
\end{figure}

We run $k$-means and $k$-pc for $k=1,\ldots,10$ and use $500$ restarts in each (much fewer would be sufficient for smaller~$k$).
The standard `elbow plot' of the cost as a function of $k$ does not reveal a clear candidate for the number of clusters, and so is omitted.
Importantly, our final goal is not to cluster the points but rather to determine the groups of concomitant extremes, hopefully leading to a sparse model.
Thus, instead of plotting the cost function, we plot the maximal face dimension in Figure~\ref{fig:nonorthogonal_model} (right) for two (angular) thresholds $\v\in\{1/10,1/5\}$,
which does suggest $k=4$ as an adequate candidate when using $k$-pc (in red). Choosing the right $k$ is difficult in general, and one may look at various statistics depending on the problem and goals of the study.
Providing some theory to aid this choice is an important open question. 

Finally, we present the detected $4$ faces in Figure~\ref{fig:nonorthogonal}, where the same two thresholds are used. Note that smaller threshold leads to larger faces and we depict additional indices by light colors.
Here and below the thresholds are chosen so that the recovered faces correspond to a sparse problem while (almost) all indices are contained in some group.
In Figure~\ref{fig:nonorthogonal}, the centroids are ordered so that the respective faces can be associated with $\{1,\ldots,40\}$, $\{21,\ldots,60\},\{1,\ldots,20,61,\ldots,100\}$ and $\{41,\ldots,100\}$.
Both methods produce relatively good results with $k$-pc being visually a bit better.
Choosing an appropriate threshold is non-trivial and this choice may depend on the ability to cope with faces of large dimension in a subsequent modelling step.
Larger threshold values result in smaller faces.
Some indices may not appear in any group indicating their high level of asymptotic independence from the rest.
In practice, it may be reasonable to first establish the main directions/prototypes of extremes and then treat the remaining components in some way.
Finally, thresholding of centroids is just one possibility to define faces, and many others exist~\citep{cha2015}.

\begin{figure}[!ht]
\centering
\includegraphics[clip, trim=20 40 0 50, width=1\textwidth]{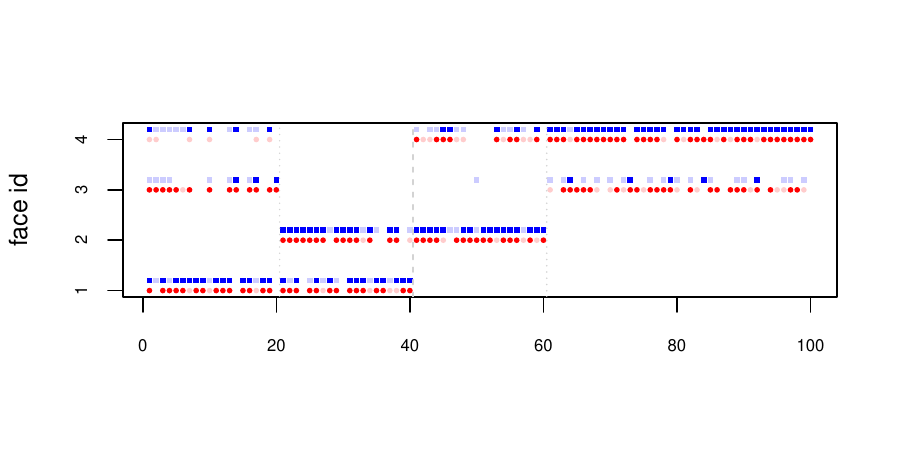}
\caption{Detection of $4$ non-orthogonal faces using $k$-means (blue) and $k$-pc (red) with $\v=1/5$ (dark colours) and $\v=1/10$ (both intensities).}
\label{fig:nonorthogonal}
\end{figure}

\subsection{River discharges}
\label{subsection_river_discharges}
We illustrate the two clustering approaches using river discharges at $d=68$ locations in Switzerland, see the review paper~\cite{eng2020_review} for a detailed description of this dataset.
The data is standardized as described in Section~\ref{sec:prelim} and then $10\%$ of observations are used to approximate the angular distribution, resulting in 202 samples. Unlike in the simulation experiments above, here we have a rather strong `asymptotic' dependence between various components, see Figure~\ref{fig:elbow} (left).
%One may attempt to identify concomitant extremes from this matrix,

%Some dependence can be explained by the geography of the measuring stations, and also by the relatively small number of observations.
%This dataset is not optimal for our purpose, since it has relatively few observations and higher degree of dependence

\begin{figure}[!ht]
\centering
\includegraphics[clip, trim=0 10 0 85, height=0.3\textwidth]{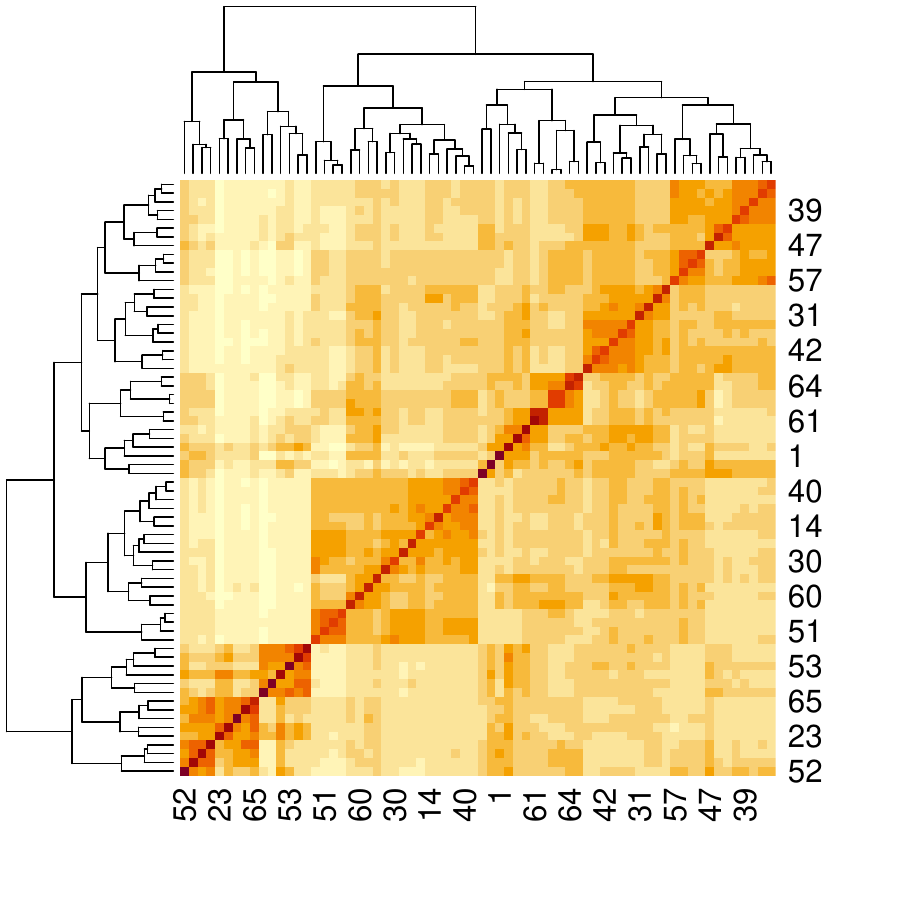}
\includegraphics[clip, trim=50 20 30 50, width=0.3\textwidth]{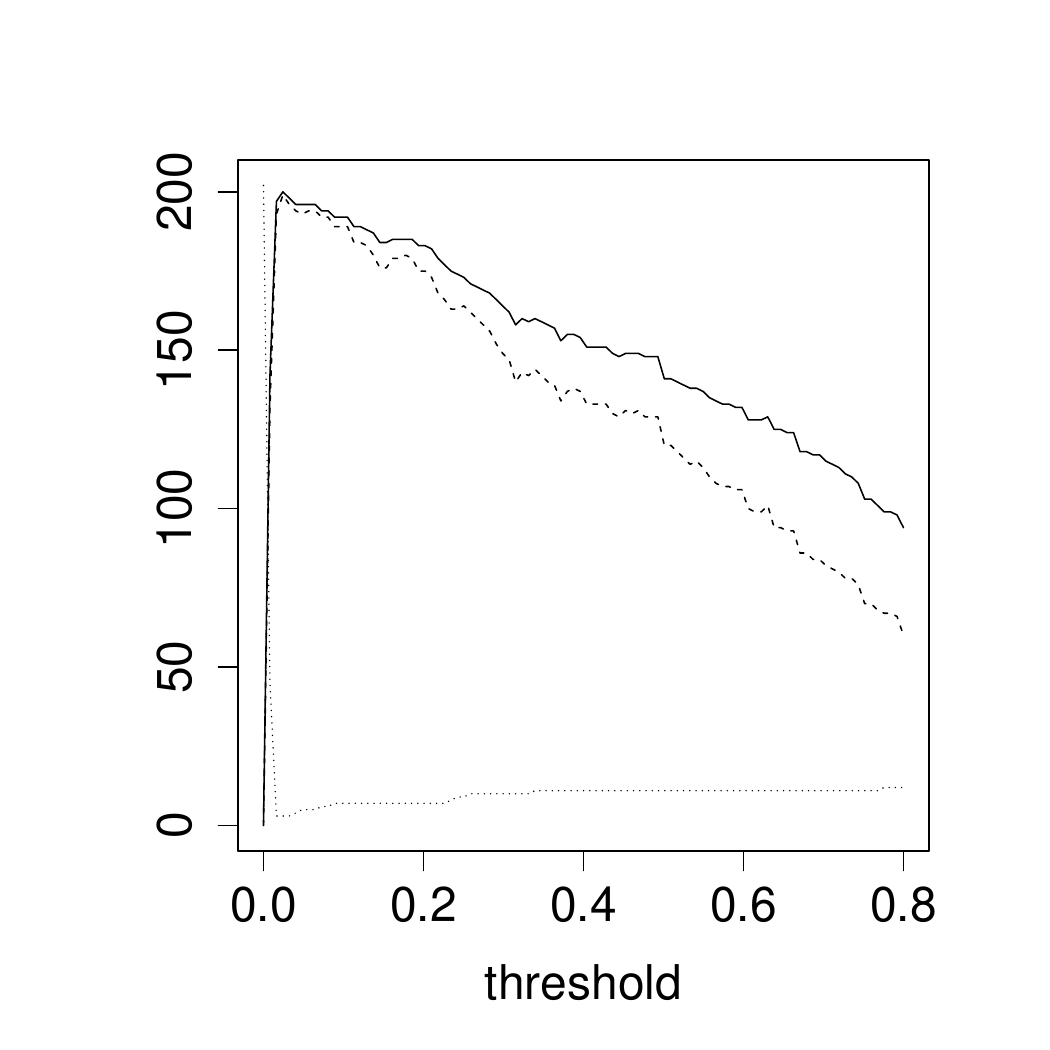}
\includegraphics[clip, trim=30 20 50 50, width=0.3\textwidth]{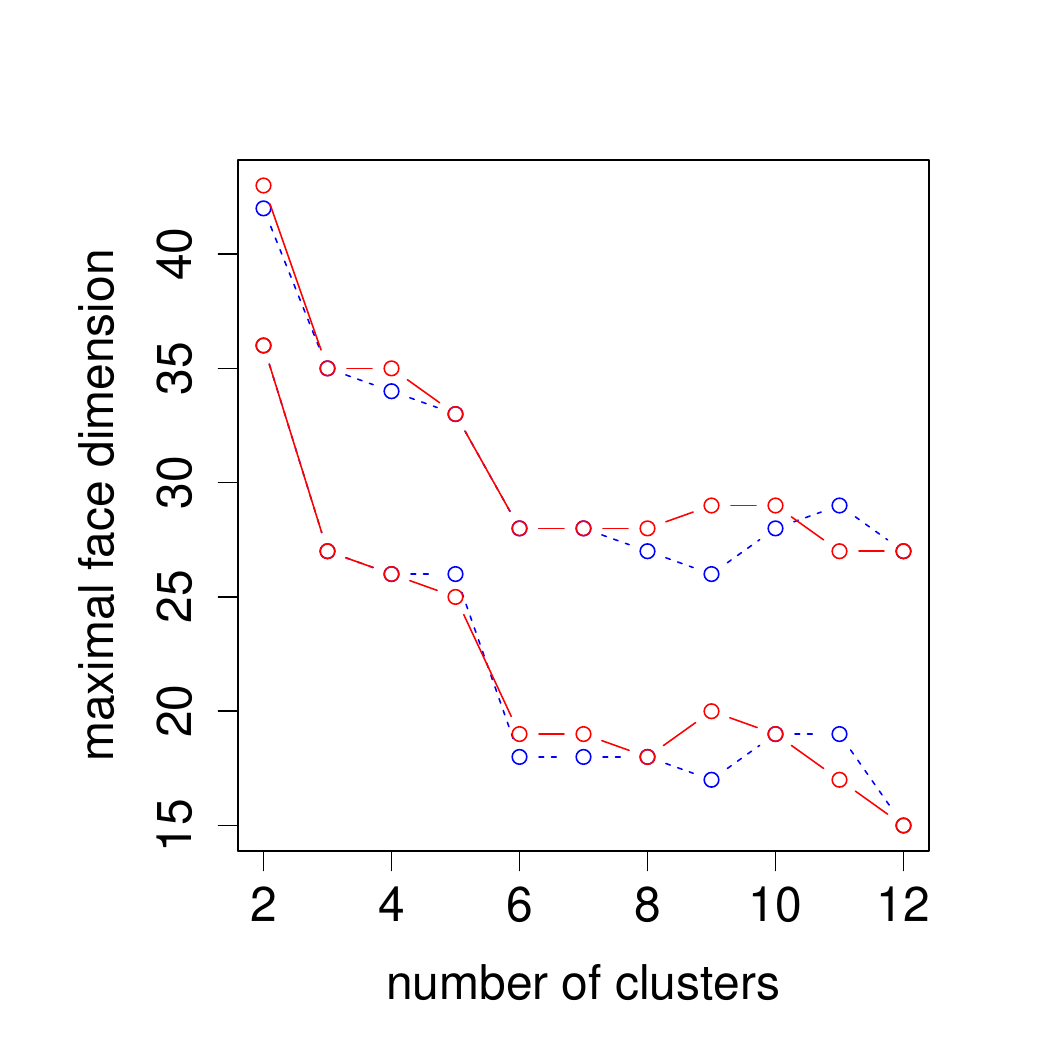}
\caption{Left: estimated $(\chi_{ij})$. Centre: truncation without clustering -- $\#$ faces (solid), $\#$ faces with a single observation (dashed), max $\#$ observations per face (dotted).
Right: clustering -- max face dimension for $\v\in\{1/4,1/3\}$.}
\label{fig:elbow}
\end{figure}

First, we investigate a basic thresholding approach without clustering as in the DAMEX algorithm~\citep{goi2016}, which assigns extreme observations (with respect to the max norm) to faces according to a chosen threshold.
Figure~\ref{fig:elbow} (centre) presents some statistics for various thresholds and also illustrates a common issue with this approach, see also~\cite{goi2017} and~\cite{chi2019} for further comments and possible solutions.
This method results in a large number of faces with a single observation, and so further grouping is needed.
Note that pruning such faces would remove a majority of the extreme observations. Upon reducing the threshold, we get one high-dimensional face containing most of the observations.
In the above example of 4 non-orthogonal faces an analogous picture is obtained.
Furthermore, similar issues (even if less pronounced) arise when using the alternative approach of~\cite{Meyer_Wintenberger_2021} in this and the previous examples.

Next, we consider clustering approaches with a relatively small number of clusters $k=2,\ldots,12$, which is desirable in practice, since every cluster leads to a submodel requiring further investigation.
%It must also be noted that clustering is a computationally hard problem (NP-hard), and the cost function in real world problems is likely to exhibit more local minima compared to the simulated example above.
%Furthermore, larger numbers of clusters normally lead to worse approximations, whereas both clustering approaches are expected to produce rather similar results.
%In order to limit the effect of suboptimal centroids on the comparison of the methods we restart each clustering procedure 30,000 times, but a much smaller number should be sufficient in practice.
In all cases the centroids produced by the two methods are very similar with the maximal angular distance of~$0.1$; here we used 30000 restarts to limit the effect of suboptimal centroids.
Figure~\ref{fig:elbow} (right) illustrates the maximal face dimension as a function of $k$, which suggests $k=6$ as an adequate number of faces.
The obtained faces using $\v=1/4$ are depicted in Figure~\ref{fig:exper6}, see also~\cite{eng2020_review} for a somewhat similar picture.
These groups of concomitant extremes exhibit nice geographical patterns.
The face dimensions are $27, 15, 4, 8, 28, 11$ (from smaller to larger circles) containing $46, 41, 19, 35, 38, 23$ extreme observations, respectively.
It is noted that the $k$-means centroids are almost the same, the cluster assignments of observations coincide, and the faces differ only slightly.
A more careful analysis may proceed by using different thresholds for each centroid, or some other face attribution procedure.

\begin{figure}[!ht]
\centering
\includegraphics[clip, trim=30 110 20 100, width=0.6\textwidth]{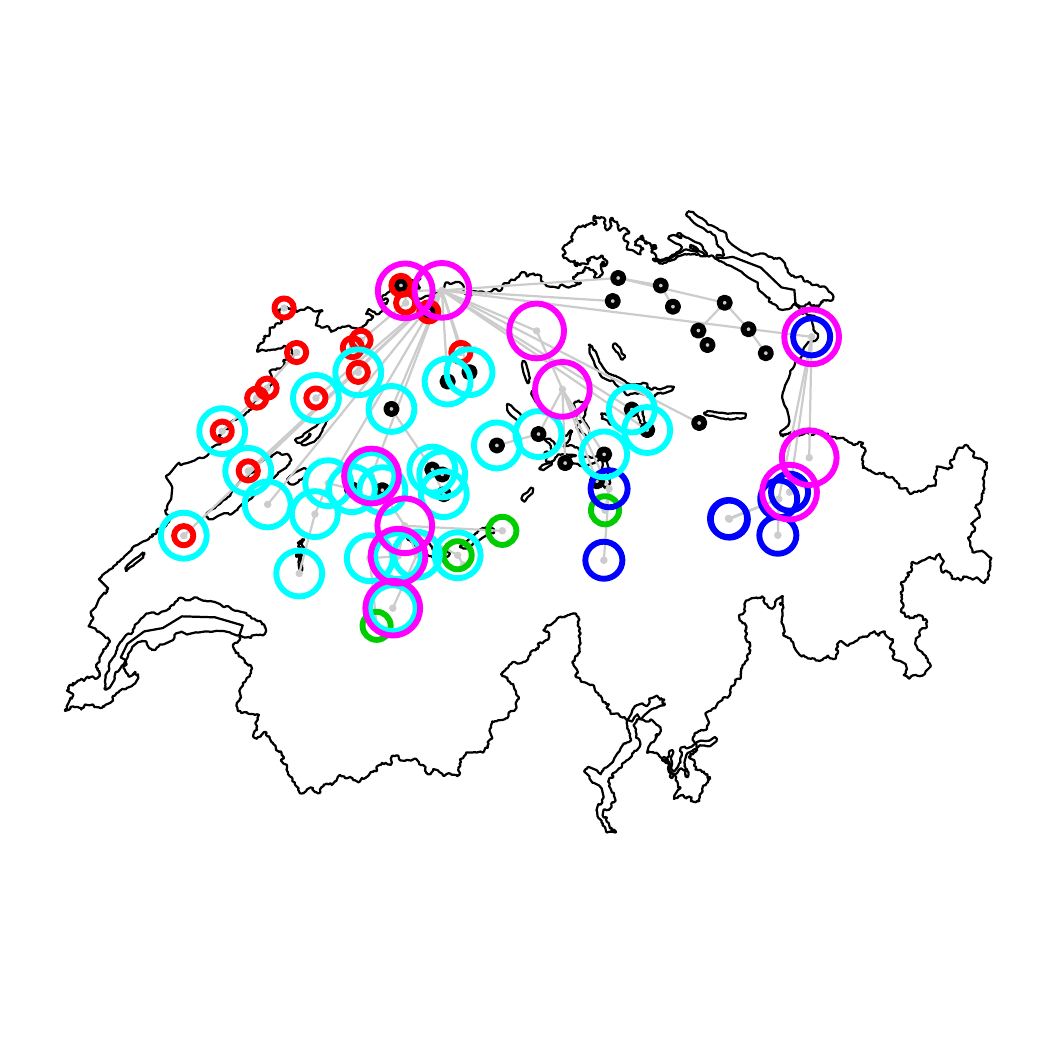}
\caption{River network with $68$ gauging stations denoted by gray points. The $k=6$ faces obtained using the $k$-pc method are depicted using circles of different color and size.}
\label{fig:exper6}
\end{figure}

In conclusion, the spherical clustering approach is a useful tool in the analysis of concomitant extremes in large dimensions. The two methods often produce similar results, but the $k$-pc method may be better in the setting when the groups exhibit weak asymptotic dependence between many pairs of their components. Importantly, it comes with a theoretical guarantee applicable in a much broader context, which is the main contribution of this work.

\section*{Acknowledgement}
The authors gratefully acknowledge financial support of Sapere Aude Starting Grant 8049-00021B ``Distributional Robustness in Assessment of Extreme Risk''.

\appendix

\section{Two Lemmas}
The first result is required for the proof of Theorem~\ref{thm:means}.

\begin{lemma}
\label{lem:triangle}
A maximizer $(\nu_1,\ldots,\nu_k)$ of
\[
\max_{\nu_1,\ldots,\nu_k\in\R^d_+} \sum_{i=1}^k \norm{\nu_i}\qquad\text{ subject to }\quad \sum_{i=1}^k \nu_i=\nu\in\R^d_+
\]
satisfies for all $i\neq j$: $\nu_i\perp\nu_j$ or both $\nu_i$ and $\nu_j$ have one positive entry at the same position. Moreover, if $k\les d$, then the maximum is attained under the assumption $\nu_i\perp\nu_j$ for all $i\neq j$.
\end{lemma}

\begin{proof}
The maximum is attained since it is the maximum of a continuous function over a compact set. Assume that for some index $\ell$ we have $a=\nu_{i\ell}>0$ and $b=\nu_{j\ell}>0$, which is equivalent to $\nu_i\nperp\nu_j$. Let $c=(\sum_{m\ne\ell} \nu^2_{im})^{1/2}\ges 0$ and $d=(\sum_{m\ne\ell} \nu^2_{jm})^{1/2}\ges 0$ be the norms when the $\ell$th coordinate is ignored. We further assume that $c+d>0$, which excludes the possibility of only one positive entry. It is enough to show that there exists a vector $u\in\R^d$ such that
\[
\norm{\nu_i+u}+\norm{\nu_j-u}>\norm{\nu_i}+\norm{\nu_j},\qquad \nu_i+u,\nu_j-u\ges 0,
\]
so that the maximality fails. We will consider the two possibilities $u=b e_\ell$ and $u=-a e_\ell$, where $e_\ell$ is the $\ell$th standard basis vector. It is clear that the positivity constraint is satisfied. Thus it is enough to show that
\begin{equation}
\label{eq:triangle}
\max\{\sqrt{(a+b)^2+c^2}+d,c+\sqrt{(a+b)^2+d^2}\}>\sqrt{a^2+c^2}+\sqrt{b^2+d^2},
\end{equation}
where $a,b>0$, $c,d\ges 0$ and $c+d>0$, which can be proved by assuming that $d>c$ so that the first entry in the maximum is larger than the second, and taking squares of both sides twice.

Finally, any group of $\nu_i$ with a single strictly positive entry in the same position can be replaced by their sum and zero vectors without changing the sum of norms.
\end{proof}

The proof of Theorem~\ref{thm:eigs} relies on the following non-standard upper bound on the sum of leading eigenvalues.

\begin{lemma}
\label{lemma:eigenvalues}
Let $M_1\ldots,M_k$ be symmetric non-negative definite matrices of order $d\ges k$. Then
\begin{equation}
\label{equation:eigenvalues}
\sum_{i=1}^k\lambda_1(M_i)\les\sum_{i=1}^k\lambda_i(M),\qquad M=\sum_{i=1}^kM_i.
\end{equation}
Moreover, for any $M$ there exist $M_i$ as above yielding the equality.
\end{lemma}

\begin{proof}
By the result of~\cite{fan1949theorem}, see also~\cite[Thm.\ 1]{overton1992sum}, we have
\begin{equation}
\label{eq:fan}
\sum_{i=1}^k \lambda_i(M)=\max\left\{\sum_{i=1}^k v_i^\top Mv_i: v_i^\top v_j=\ind{i=j}\right\}
\end{equation}
and, in particular, $\lambda_1(M)=\max\{v^\top Mv: \norm{v}=1\}$.

Let vectors $u_1,\ldots,u_k$ have unit norms and be such that $\lambda_1(M_i)=u_i^\top M_iu_i$, and let $v_1,\ldots,v_k$ be mutually orthogonal vectors with unit norms obtained via the Gram--Schmidt process. That is, for all $i=1,\ldots,k$ we have $u_i=\sum_{j=1}^i c_{ij}v_{j}$, where $\sum_{j=1}^i c_{ij}^2=1$.

By~\eqref{eq:fan} we have
\[
\sum_{j=1}^k \lambda_j(M)\ges\sum_{j=1}^k v_j^\top M v_j=\sum_{i=1}^k \sum_{j=1}^k v_j^\top M_iv_j\ges\sum_{i=1}^k \sum_{j=1}^i v_j^\top M_iv_j,
\]
where in the last step we use the assumption that $M_i$ are non-negative definite. Furthermore, for any $i$ we have the representation
\[
\lambda_1(M_i)=u_i^\top M_iu_i=\left(\sum_{j=1}^i c_{ij}v_j^\top\right)M_i\left(\sum_{j=1}^i c_{ij}v_j\right)=\sum_{j=1}^i c_{ij}^2v_j^\top M_iv_j+\sum_{j\les i} \sum_{j\neq\ell\les i} c_{ij}c_{i\ell}v_j^\top M_iv_\ell.
\]
Thus, we find a lower bound on the difference of interest:
\[
\sum_{i=1}^k \lambda_i(M)-\sum_{i=1}^k \lambda_1(M_i)\ges\sum_{i=1}^k\left(\sum_{j=1}^i (1-c_{ij}^2)v_j^\top M_iv_j-\sum_{j\les i} \sum_{j\neq\ell\les i} c_{ij}c_{i\ell}v_j^\top M_iv_\ell\right).
\]
But according to the constraints on $c_{ij}$ we have
\[
\sum_{j=1}^i (1-c_{ij}^2)v_j^\top M_iv_j=\sum_{j\les i} \sum_{j\neq\ell\les i} c^2_{i\ell}v_j^\top M_iv_j,
\]
and it is left to check non-negativity of the terms
\[
\sum_{j\les i} \sum_{j\neq \ell\les i} (c^2_{i\ell}v_j^\top M_iv_j-c_{ij}c_{i\ell}v_j^\top M_iv_\ell)=\sum_{1\les\ell<j\les i}(c_{i\ell}v_j-c_{ij}v_\ell)^\top M_i(c_{i\ell}v_j-c_{ij}v_\ell),
\]
which is true due to assumed non-negative definiteness of all $M_i$, and the first assertion is proven.

The second assertion is obtained by taking the matrices
\begin{gather*}
M_i=Q\diag(0,\ldots,0,\lambda_i(M),0,\ldots,0)Q^\top,\quad 1\les i\les k-1,\\
M_k=Q\diag(0,\ldots,0,\lambda_k(M),\lambda_{k+1}(M),\ldots,\lambda_d(M))Q^\top,
\end{gather*}
where $Q$ is an orthogonal matrix in the diagonalization of $M$. Indeed, the matrices $M_i$ are non-negative definite summing up to $M$, and the largest eigenvalues are $\lambda_i(M)$ for $i\les k$.
\end{proof}

\section{H\"usler--Reiss distribution}
\subsection{Random generation of variogram}
A $d$-dimensional max-stable H\"usler--Reiss distribution~\citep{Husler1989} is parametrized by a conditionally negative definite $d\times d$ matrix $\Gamma$, called a variogram. Every such matrix has a representation $\Gamma_{ij}=\norm{h_i-h_j}^2$, where $h_1,\ldots,h_d$ are the elements of some Hilbert space with norm $\norm{\cdot}$, and every such construction yields a conditionally negative definite matrix, see~\cite[Property~(g), p.~191]{Vakhania_Tarieladze_Chobanyan}. As usual, we impose a further non-degeneracy assumption that $\Gamma$ is strictly conditionally negative definite so that the associated exponent measure has a density~\citep{Engelke2015}.

Suppose that $Y$ has a max-stable H\"usler--Reiss distribution. Our main assumption requires a partition $I_1\cup I_2=\{1,\ldots,d\}$ of the index set such that $\chi_{ij}=\chi_{ji}=0$ for all $i\in I_1$ and $j\in I_2$, which is equivalent to (asymptotic) independence of the two groups of components of~$Y$. Recall the formula for the bivariate tail dependence coefficient:
\[
\chi_{ij}=2\overline\Phi(\sqrt{\Gamma_{ij}}/2),
\]
where $\overline\Phi$ is the survival function of the standard normal distribution, and note that $\chi_{ij}=0$ is obtained in the limit case $\Gamma_{ij}=\infty$. In practice, large entries in the respective locations of $\Gamma$ would  suffice, because of the weak convergence of the respective distributions.

In our experiments we randomly generate $\Gamma$ for $d=100$ according to the following procedure based on the aforementioned representation. Firstly, we generate $d$-dimensional vectors $h_1,\ldots,h_d$ with i.i.d.\ Pareto distributed components having shape parameter~$2.5$. Secondly, we add one more dimension by letting $\widetilde{h}_i=(h_i,L\ind{i\in I_1})$, where $L=10^5$ is a fixed large number. Then the variogram matrix $\Gamma$ is set to $\frac{3}{d}(\norm{\widetilde{h}_i-\widetilde{h}_j}^2)_{ij}$, where $3/d$ provides a scaling resulting in a suitable distribution of the tail dependence coefficients. Note that this procedure ensures that $\Gamma_{ij}\geqslant L$ for $i$ and $j$ in different groups.

Finally, we use the R package~\citep{graphicalExtremes} to sample vectors $Y$ from the H\"usler--Reiss distribution specified by~$\Gamma$. Note that $Y/\norm{Y}$ given $\norm{Y}>t$ for a finite threshold $t$ provides only an approximation of the limiting angle. The distribution of the exact angle is addressed below.

\subsection{The angular distribution and the matrix $\Sigma$}
Even though the exact matrix $\Sigma=\e(XX^\top)$ is not required for the simulation experiments, it can be used to verify the sufficient conditions in our main result. Below we determine the probability density function of $X$ (with respect to the Euclidean norm) and provide an expression for $\Sigma$ in terms of expectations of transformed Gaussian vectors.

As in~\cite{Engelke2015}, we consider the $(d-1)\times(d-1)$ covariance matrix
\[
R=\left\{\frac{1}{2}(\Gamma_{id}+\Gamma_{jd}-\Gamma_{ij})\right\}_{i,j\neq d}.
\]
For all $i=1,\ldots,d$ we define the transformation
\begin{equation}
\label{eq:t}
t_i(x)=\log(x_i/x_d)=\log\left(x_i/\sqrt{1-x_1^2-\cdots-x_{d-1}^2}\right),\qquad x\in\mathrm{int}(\Sp^{d-1}),
\end{equation}
where $\mathrm{int}(\Sp^{d-1})$ stands for the simplex interior, and note that $t\colon\mathrm{int}(\Sp^{d-1})\to\R^{d-1}$ is a bijection with the inverse
\[
t^{-1}(z)=\frac{(\exp(z_1),\ldots,\exp(z_{d-1}),1)^\top}{\sqrt{1+\exp(2 z_1)+\cdots+ \exp(2z_{d-1})}}.
\]
It turns out that each column of the matrix $\Sigma/\mu$ corresponds (up to a permutation) to $\e(t^{-1}(Z))$ for a multivariate normal~$Z$ as specified below. The means and other moments of $X$ can be obtained in a similar way.

\begin{lemma}
For the $d$-dimensional H\"usler--Reiss distribution with the variogram $\Gamma$ the density of the Euclidean angle $(X_1,\ldots,X_{d-1})$ is given by
\begin{equation}
f(x_1,\ldots,x_{d-1})=\mu\, x^{-2}_d\prod_{i\leqslant d} x^{-1}_i\,\phi_{d-1}(t_1(x),\ldots, t_{d-1}(x)),\qquad x\in\mathrm{int}(\Sp^{d-1}),
\end{equation}
where $\mu=\e(X_1)$ and $\phi_{d-1}$ is the density of the multivariate normal distribution with covariance matrix $R$ and mean vector $-(\Gamma_{1d},\ldots,\Gamma_{d-1,d})^\top/2$. Furthermore, with $g$ being the density of $(W_1,\ldots,W_{d-1})$, where $W=t^{-1}(Z)\in\Sp^{d-1}$ and $Z\sim\phi_{d-1}$, the following identities hold:
\begin{gather*}
f(x_1,\ldots,x_{d-1})=\mu g(x_1,\ldots,x_{d-1})/x_d,\\
\mu^{-1}=\e(W_d^{-1}),\quad\e(X_iX_d)=\mu\e(W_i),\quad i=1,\ldots,d.
\end{gather*}
\end{lemma}

\begin{proof}
From~\cite{Engelke2015} it is known that the exponent measure has the density
\[
\lambda(y)=y^{-1}_d\prod_{i\leqslant d} y^{-1}_i\,\phi_{d-1}(\log(y_1/y_d),\ldots, \log(y_{d-1}/y_d)),\qquad y_1,\ldots,y_d>0.
\]
We note that for the polar coordinates
\[
y_1=rx_1,\quad y_2=rx_2,\quad\ldots,\quad y_d=r(1-x_1^2-\ldots-x_{d-1}^2)^{1/2}
\]
the absolute value of the Jacobian evaluates to $r^{d-1}/x_d$, and so in these coordinates the density becomes
\[
r^{-2}x^{-2}_d\prod_{i\leqslant d} x^{-1}_i\,\phi_{d-1}(t_1(x),\ldots,t_{d-1}(x)).
\]
This must factorize into $r^{-2}f(x_1,\ldots,x_{d-1})/\mu$ according to~\cite[Prop.\ 5.11(iv)]{res2008} and the first result follows.

Let us compute the Jacobian $J_t(x)$ of the transformation $t(x)$ defined in~\eqref{eq:t} and restricted to the first $d-1$ values. We have
\[
\dfrac{\partial t_i}{\partial x_j}(x)=\dfrac{x_j}{s},\quad j\neq i,\quad\text{and}\quad \dfrac{\partial t_i}{\partial x_i}(x)=\dfrac{x_i^2+s}{x_is},
\]
where $s=1-x_1^2-\ldots-x_{d-1}^2=x_d^2$, and so
\[
J_t(x)=
\begin{vmatrix}
\dfrac{x_1^2+s}{x_1s} & \dfrac{x_2}{s} & \ldots & \dfrac{x_{d-1}}{s}\\
\dfrac{x_1}{s} & \dfrac{x_2^2+s}{x_2s} & \ldots & \dfrac{x_{d-1}}{s}\\
\vdots & \vdots & \ddots & \vdots\\
\dfrac{x_1}{s} & \dfrac{x_2}{s} & \ldots & \dfrac{x_{d-1}^2+s}{x_{d-1}s}\\
\end{vmatrix}
=\dfrac{x_1\ldots x_{d-1}}{s^{d-1}}
\begin{vmatrix}
\dfrac{s}{x_1^2}+1 & 1 & \ldots & 1\\
1 & \dfrac{s}{x_2^2}+1 & \ldots & 1\\
\vdots & \vdots & \ddots & \vdots\\
1 & 1 & \ldots & \dfrac{s}{x_{d-1}^2}+1\\
\end{vmatrix}
.
\]
Using formula~(1) in~\cite{Goberstein}, we find that the last determinant is equal to
\[
\dfrac{s^{d-1}}{x_1\ldots x_{d-1}}+\dfrac{s^{d-1}}{x_1\ldots x_{d-1}}\cdot \left(\dfrac{x_1^2}{s}+\ldots+\dfrac{x_{d-1}^2}{s}\right)=\dfrac{1}{x_1\ldots x_{d-1}s}.
\]
Thus, we have the relation $g(x)=(x_1\ldots x_{d-1}s)^{-1}\phi_{d-1}(t(x))$, where $x$ and $t(x)$ are restricted to the first $d-1$ elements. Comparing it with $f$, we indeed find the stated relation between $f$ and~$g$. Finally, observe that
\[
\e(X_iX_d)=\mu\int x_ig(x)\,dx=\mu\e(W_i),\qquad 1=\mu\e(W^{-1}_d).
\]
The proof is complete.
\end{proof}

The above result states that the last column of $\Sigma/\mu$ is given by the vector $\e(W)$. By changing the indexing we may easily find the other columns. In particular, in the bivariate case with parameter $\gamma=\Gamma_{12}$ we have
\[
W_1=\left[1+\exp\left(-2\sqrt\gamma Z_0+\gamma)\right)\right]^{-1/2},\qquad W_2=(1-W_1^2)^{1/2}=\left[1+\exp\left(2\sqrt\gamma Z_0-\gamma)\right)\right]^{-1/2},
\]
where~$Z_0$ is a standard normal. Sadly, the moments $w_1=\e(W_1)$, $w_2=\e(W_2)$ are not explicit, and the same applies to the matrix~$\Sigma/\mu=\left(\begin{smallmatrix} 1-w_2 & w_1\\ w_1 & w_2\end{smallmatrix}\right)$.


\begin{thebibliography}{}

\bibitem[Chautru, 2015]{cha2015}
Chautru, E. (2015).
\newblock Dimension reduction in multivariate extreme value analysis.
\newblock {\em Electron. J. Stat.}, 9(1):383--418.

\bibitem[Chiapino et~al., 2020]{chi2019}
Chiapino, M., Cl{\'e}men{\c{c}}on, S., Feuillard, V., and Sabourin, A. (2020).
\newblock A multivariate extreme value theory approach to anomaly clustering
  and visualization.
\newblock {\em Computational Statistics}, 35:607--628.

\bibitem[Cooley and Thibaud, 2019]{coo2019}
Cooley, D. and Thibaud, E. (2019).
\newblock Decompositions of dependence for high-dimensional extremes.
\newblock {\em Biometrika}, 106(3):587--604.

\bibitem[Davison and Huser, 2015]{dav2015}
Davison, A. and Huser, R. (2015).
\newblock Statistics of extremes.
\newblock {\em Ann. Rev. Stat. App.}, 2:203--235.

\bibitem[Dhillon and Modha, 2001]{dhi2001}
Dhillon, I.~S. and Modha, D.~S. (2001).
\newblock Concept decompositions for large sparse text data using clustering.
\newblock {\em Machine learning}, 42(1-2):143--175.

\bibitem[Drees and Sabourin, 2021]{dre2019}
Drees, H. and Sabourin, A. (2021).
\newblock Principal component analysis for multivariate extremes.
\newblock {\em Electron. J. Statist.}, 15(1):908--943.

\bibitem[Engelke and Hitz, 2020]{eng2018}
Engelke, S. and Hitz, A.~S. (2020).
\newblock Graphical models for extremes.
\newblock {\em J. R. Stat. Soc. B.}, 82(4):871--932.

\bibitem[Engelke et~al., 2019]{graphicalExtremes}
Engelke, S., Hitz, A.~S., and Gnecco, N. (2019).
\newblock {\em graphical{E}xtremes: Statistical Methodology for Graphical
  Extreme Value Models}.
\newblock Available from
  \texttt{https://CRAN.R-project.org/package=graphicalExtremes}, R package
  version 0.1.0.

\bibitem[Engelke and Ivanovs, 2021]{eng2020_review}
Engelke, S. and Ivanovs, J. (2021).
\newblock Sparse structures for multivariate extremes.
\newblock {\em Annu. Rev. Stat. Appl.}, 8(1):241--270.

\bibitem[Engelke et~al., 2015]{Engelke2015}
Engelke, S., Malinowski, A., Kabluchko, Z., and Schlather, M. (2015).
\newblock Estimation of {H}{\"u}sler--{R}eiss distributions and
  {B}rown--{R}esnick processes.
\newblock {\em J. R. Stat. Soc. B.}, 77(1):239--265.

\bibitem[Fan, 1949]{fan1949theorem}
Fan, K. (1949).
\newblock On a theorem of {W}eyl concerning eigenvalues of linear
  transformations {I}.
\newblock {\em Proceedings of the NAS of the USA}, 35(11):652.

\bibitem[Gan et~al., 2007]{clustering_book}
Gan, G., Ma, C., and Wu, J. (2007).
\newblock {\em Data clustering}, volume~20.
\newblock SIAM and ASA.
\newblock Theory, algorithms, and applications.

\bibitem[Goberstein, 1980]{Goberstein}
Goberstein, S.~M. (1980).
\newblock Evaluating "uniformly filled" determinants.
\newblock {\em Col. Math. J.}, 19(4):343--345.

\bibitem[Goix et~al., 2016]{goi2016}
Goix, N., Sabourin, A., and Cl\'{e}men\c{c}on, S. (2016).
\newblock Sparse representation of multivariate extremes with applications to
  anomaly ranking.
\newblock In {\em Proceedings of the 19th International Conference on
  Artificial Intelligence and Statistics (AISTATS)}. JMLR: W\&CP.

\bibitem[Goix et~al., 2017]{goi2017}
Goix, N., Sabourin, A., and Clémençon, S. (2017).
\newblock Sparse representation of multivariate extremes with applications to
  anomaly detection.
\newblock {\em J. Multivar. Anal.}, 161:12--31.

\bibitem[Hill et~al., 2013]{Hill_et_al}
Hill, M.~O., Harrower, C.~A., and Preston, C.~D. (2013).
\newblock Spherical $k$-means clustering is good for interpreting multivariate
  species occurrence data.
\newblock {\em Methods Ecol. Evol.}, 4(6):542--551.

\bibitem[H{\"u}sler and Reiss, 1989]{Husler1989}
H{\"u}sler, J. and Reiss, R.-D. (1989).
\newblock Maxima of normal random vectors: between independence and complete
  dependence.
\newblock {\em Stat. Prob. Let.}, 7(4):283--286.

\bibitem[Jalalzai and Leluc, 2020]{jalalzai2020informative}
Jalalzai, H. and Leluc, R. (2020).
\newblock Informative clusters for multivariate extremes.
\newblock {\em arXiv:2008.07365}.

\bibitem[Janssen and Wan, 2020]{Janssen_Wan}
Janssen, A. and Wan, P. (2020).
\newblock $k$-means clustering of extremes.
\newblock {\em Electron. J. Stat.}, 14:1211--1233.

\bibitem[Larsson and Resnick, 2012]{lar2012}
Larsson, M. and Resnick, S.~I. (2012).
\newblock Extremal dependence measure and extremogram: the regularly varying
  case.
\newblock {\em Extremes}, 15:231--256.

\bibitem[Meyer and Wintenberger, 2021]{Meyer_Wintenberger_2021}
Meyer, N. and Wintenberger, O. (2021).
\newblock Sparse regular variation.
\newblock {\em Advances in Applied Probability}, 53(4):1115–1148.

\bibitem[Overton and Womersley, 1992]{overton1992sum}
Overton, M.~L. and Womersley, R.~S. (1992).
\newblock On the sum of the largest eigenvalues of a symmetric matrix.
\newblock {\em SIMAX}, 13(1):41--45.

\bibitem[Palarea-Albaladejo et~al., 2012]{palarea2012dealing}
Palarea-Albaladejo, J., Mart{\'\i}n-Fern{\'a}ndez, J.~A., and Soto, J.~A.
  (2012).
\newblock Dealing with distances and transformations for fuzzy $c$-means
  clustering of compositional data.
\newblock {\em J. Classif.}, 29(2):144--169.

\bibitem[Resnick, 2008]{res2008}
Resnick, S.~I. (2008).
\newblock {\em Extreme Values, Regular Variation and Point Processes}.
\newblock Springer, New York.

\bibitem[Simpson et~al., 2020]{sim2018}
Simpson, E.~S., Wadsworth, J.~L., and Tawn, J.~A. (2020).
\newblock Determining the dependence structure of multivariate extremes.
\newblock {\em Biometrika}, 107(3):513--532.

\bibitem[Vakhania et~al., 1980]{Vakhania_Tarieladze_Chobanyan}
Vakhania, N.~N., Tarieladze, V.~I., and Chobanyan, S.~A. (1980).
\newblock {\em Probability Distributions on Banach Spaces}.
\newblock Mathematics and Its Applications (Soviet Series). D. Reidel
  Publishing Company.

\bibitem[Wang et~al., 2018]{wang2018efficient}
Wang, J., Wang, W., Garber, D., and Srebro, N. (2018).
\newblock Efficient coordinate-wise leading eigenvector computation.
\newblock In {\em Algorithmic Learning Theory}, pages 806--820. PMLR.

\end{thebibliography}
\end{document}